\documentclass[letterpaper, 10 pt, conference]{ieeeconf}  

\IEEEoverridecommandlockouts                              

\usepackage{cite}
\usepackage{color}
\usepackage{xxcolor}
\usepackage{mathrsfs}
\usepackage{subcaption}
\usepackage{verbatim}
\usepackage{bbold}
\usepackage{tikz}
\usetikzlibrary{calc,intersections}

\usepackage{graphicx}
\usepackage{tabularx}
\usepackage{relsize}
\usepackage{exscale}

\usepackage[cmex10]{amsmath}
\usepackage{array}
\usepackage{algorithmic}
\usepackage{amsfonts}
\usepackage[ruled, vlined]{algorithm2e}
\usepackage{url}
\usepackage{enumerate}
\usepackage{tikz}
\usepackage{pgfplots}
\usepackage{mathtools}
\usepackage{pgfmath}
\usepgfplotslibrary{groupplots}
\usepgfplotslibrary{fillbetween}
\usetikzlibrary{pgfplots.fillbetween}

\let\labelindent\relax
\usepackage{enumitem}
\usepackage{amsthm}
\usepackage{subcaption}   
\usepackage{comment}
\usepackage{xcolor}

\usepackage[draft]{todonotes} 

\newcommand{\kk}{{\kappa_{\tau}}}

\newcommand{\dd}{{\delta}}
\newcommand{\tnull}{{\tau}}

\DeclareMathOperator*{\argmin}{argmin}

\newcommand{\medcup}{\mathsmaller{\bigcup}}

\newcommand{\cl}[1]{\mathcal{#1}}

\newcommand{\lemref}[1]{Lemma \ref{#1}}
\newcommand{\thmref}[1]{Theorem \ref{#1}}

\newcommand{\corref}[1]{Corollary\ref{#1}}

\newcommand{\secref}[1]{Section \ref{#1}}

\newcommand{\figref}[1]{Fig. \ref{#1}}
\newcommand{\defref}[1]{Def. \ref{#1}}

\newcommand{\R}{\mathbb{R}}

\newcommand{\anon}{\cdot}
\newcommand{\CC}{\bm{K}^{\mathrm{cc}}}
\newcommand{\CCt}[1]{K^{\mathrm{cc}}_{#1}}

\newcommand{\s}[1]{\mathsf{#1}}
\newcommand{\bm}[1]{\mathbf{#1}}
\renewcommand{\c}[1]{\mathcal{#1}}

\newcommand{\cx}[1]{\mathrm{c}(#1)}
\newcommand{\eps}{\varepsilon}

\newcommand{\Nxy}[2]{{\| #1 \|}_{#2}}

\newcommand{\BXt}[1]{\Nxy{\s{W}}{\s{X}_{#1}}}
\newcommand{\BXtnull}{\BXt{\tnull-1}}
\newcommand{\BXtneg}{\BXt{-1}}
\newcommand{\vol}{\mathrm{Vol}}
\newcommand{\dxyK}[3]{d(#1,#2;#3)}

\newcommand{\rad}{\dd}
\usepackage[draft]{todonotes} 

\excludecomment{maybe}
\includecomment{normal}
\excludecomment{extended}

\theoremstyle{plain}
\newtheorem{thm}{Theorem}[section]
\newtheorem*{thm*}{Theorem}
\newtheorem{lem}[thm]{Lemma} 
\newtheorem{coro}{Corollary}[thm]

\theoremstyle{definition}
\newtheorem{defn}{Definition}[section] 
\newtheorem*{defn*}{Definition} 

\newtheorem{rem}{Remark}[thm]

\newtheorem*{rem*}{Remark}

\theoremstyle{remark}
\newtheorem*{ex}{Example}

\title{\LARGE \bf
 Robust Model-Free Learning and Control\\ without Prior Knowledge
}

\author{Dimitar Ho and John C. Doyle
\thanks{Dimitar Ho and John C. Doyle are with the Department of Computing and Mathematical Sciences, California Institute of Technology, Pasadena, CA. 
        {\tt\small dho@caltech.edu, doyle@caltech.edu}}%
}

\begin{document}
\maketitle
\thispagestyle{empty}
\pagestyle{empty}


%

\begin{abstract}
We present a simple model-free control algorithm that is able to robustly learn and stabilize an unknown discrete-time linear system with full control and state feedback subject to arbitrary bounded disturbance and noise sequences. The controller does not require any prior knowledge of the system dynamics, disturbances, or noise, yet it can guarantee robust stability and provides asymptotic and worst-case bounds on the state and input trajectories. To the best of our knowledge, this is the first model-free algorithm that comes with such robust stability guarantees without the need to make any prior assumptions about the system. We would like to highlight the new convex geometry-based approach taken towards robust stability analysis which served as a key enabler in our results. We will conclude with simulation results that show that despite the generality and simplicity, the controller demonstrates good closed-loop performance.

\end{abstract}

\section{INTRODUCTION}
\subsection{Motivation and Problem Statement}
\noindent Learning to stabilize unknown dynamical systems from online data has been an active research area in the control community since the 1950s  \cite{1950survey} and has recently attracted the attention of the machine learning community, foremost in the context of reinforcement learning. Although there has been extensive research on this topic, very few of the developed algorithms have reached the level of adoption in real world applications as one would expect. Particularly in areas where frequent interaction with the physical world is necessary, system failure is costly, and the deployment of control algorithms is only possible if the algorithm can guarantee that minimal safety and performance specifications will be met during operation. Although there have been past research \cite{94astromadaptive},\cite{aastrom1973self} and recent research efforts \cite{garcia2015comprehensive},\cite{dean2018regret}, \cite{abbasi2018regret}, \cite{berkenkamp2017safe}, \cite{cohen2019learning},\cite{SafeLQRDean2019} to address this problem, very few algorithms come with the necessary performance and safety guarantees to be deployed in real world applications thus far. Motivated by this, we revisit the basic problem of learning to stabilize a linear system and aim to find learning and control strategies with the least restrictive assumptions that can still give robust stability bounds for the closed loop.

In this paper, we focus on the problem of adaptively stabilizing a linear discrete-time system
\begin{subequations} \label{eq:sys}
\begin{align}
 z_{k+1} &= A_0z_k + u_k + d_k\\
x_k &= z_k + n_k
\end{align}
\end{subequations}
with state $z_k$, bounded disturbance $d_k$, bounded noise $n_k$ and control action $u_k$ that is only allowed to depend on noisy state measurements until time $k$, i.e: $x_{0},\dots ,x_{k}$. We are interested in finding controllers that can stabilize (in the sense of BIBO- or input-to-state stability guarantees) without requiring any additional assumptions about the unknown system matrix $A_0$ and the disturbance/noise sequences $(d_k)$, $(n_k)$. While admittedly system \eqref{eq:sys} describes a very restrictive class of linear systems (full state feedback and control), nearly all available learning and control approaches need to make some prior assumptions about this system in order to state stability and performance guarantees. Most commonly, these assumptions come in the form of a priori bounds on $d_k$, $n_k$ and/or $A_0$.
\subsection{Related Work}
\noindent We will review relevant literature in the context of our problem setting. 
Classical control approaches are found in the literature of adaptive control with \cite{ioannou1996robust},\cite{Ioannou:2006}, \cite{sastry2011adaptive} focusing on the deterministic and \cite{94astromadaptive} on the stochastic setting. The self-tuning regulator \cite{aastrom1973self} and its variations come with asymptotic optimality \cite{guo1995convergence}, yet robust stability guarantees without restrictive assumptions are few and can only be made in the probabilistic sense. On the deterministic side \cite{ioannou1996robust},\cite{Ioannou:2006} point out that instabilities can occur with traditional adaptive schemes and provide improved version of adaptive controllers that come with robust stability and performance guarantees. Yet, the desired guarantees depend on knowing some bounds of the system parameters and disturbance signals. 
Other challenges associated with classical adaptive control approaches are discussed in \cite{anderson2008challenges}, \cite{anderson2005failures}. Methods in safe reinforcement learning \cite{garcia2015comprehensive}, \cite{berkenkamp2017safe},\cite{akametalu2014reachability}, \cite{fisac2018general} have made great progress towards methods that guarantee robust safety properties for classes of nonlinear systems, yet the synthesis procedures involved are computationally expensive and they require knowledge of an initially robust stabilizing controller, even in the case of the simple linear system \eqref{eq:sys}. Recent work \cite{abbasi2018regret},\cite{dean2017sample},\cite{dean2018regret}, \cite{mania2019certainty}, \cite{cohen2019learning}, \cite{SafeLQRDean2019} has made significant progress in providing algorithms with robust finite-time performance guarantees for the adaptive linear quadratic gaussian regulator problem. 
However, in the context of our simple linear problem setting, all methods require that the uncertainty in the system dynamics (i.e. $A_0$) is small enough at the outset in order to provide stability guarantees of the closed loop.


\subsection{Main contribution and overview of the paper}
 \noindent In this work, we present a simple controller that can adaptively stabilize \eqref{eq:sys} without any further assumptions on disturbance, noise or the system matrix $A_0$. The presented algorithm performs tractable computations (solving a linear program each time step) and provides both uniform asymptotic and worst-case guarantees on the state- and input trajectories. An additional surprising feature of the presented algorithm is that it is not based on the certainty-equivalence principle and has a completely model-free formulation. To the best of our knowledge, this is the first model-free adaptive controller that can give our robust stability guarantees without requiring any prior knowledge about the unknown system \eqref{eq:sys}.

Our core theoretical contribution is a novel approach towards stability analysis. We first show that in any closed loop trajectory $(x_t)$, there are only a finite number of time-instances $t_i$ at which $x_{t_i+1}$ is significantly larger than $x_{t_i}$. We term those time-instances as "unstable transitions" and our first main theorem shows an upper-bound on the occurrence of these unstable transitions in the closed loop state trajectory. Then, our stability and performance guarantees follow as corollaries of this result.

We develop a new technique based on convex geometry to bound the occurrence of unstable transitions in the closed loop. Vaguely speaking, our main idea is to show that if an unstable transition occurs at some time $t'$, our proposed adaptive controller learns enough from this observation to prevent similar unstable transitions from occurring in the future. Mathematically, we formulate this idea in two steps: 1. We define a distance function $d$ between unstable transitions and show that w.r.t. to $d$, we can identify the set of unstable transitions with a bounded separated set $\s{P}$ of equal cardinality. 2. We bound the cardinality of $\s{P}$ by a metric-entropy type of quantity, which leads to an upper bound on the maximum number of times that unstable transitions can occur. We discuss the convex geometry based techniques in detail, to emphasize their potential use for robust design and analysis of learning and control algorithms, particularly in the model-free setting.

The paper is organized as follows. We formulate our problem in \secref{sec:prob} and give a brief overview of our main results in \secref{sec:previewresult}. In \secref{sec:mfadap}, we derive the model free closed loop equation and explain the intuition behind the proposed control law. In \secref{sec:mainresult}, we present and discuss our main results in detail. \secref{sec:stabilityanalysis} and \secref{sec:boundingproof} highlight the main techniques and ideas used to prove our results. \secref{sec:boundingproof} and the supporting discussion in the appendix discuss the convex geometry based techniques used to establish the bound on the unstable closed loop transitions. \secref{sec:metricentropy} highlights a parallel between the role of metric entropy in the context of our results and in the context of statistical learning theory, which could serve as an interesting intersection for future research. We conclude with some experimental results in \secref{sec:sim}. 
\section{Problem Setup}\label{sec:prob}
\noindent For our discussion, we transform the system \eqref{eq:sys} w.r.t to the measurements $x_t$ and obtain the equivalent\footnote{Since $x_t = z_t+n_t$, controlling the system state $z_t$ is equivalent to controlling the noisy measurement $x_t$.} system
        \begin{align}
                \label{eq:syss} x_{t+1} &= A_0 x_t + u_t + w_t\\
        \notag w_t &:= d_t +n_{t+1}-A_0 n_{t},
        \end{align}
where $w_t$ represents the lumped bounded disturbance at time $t$ which summarizes the influence on the system of the original noise and disturbance. A causal controller can be represented as a collection $\bm{K} =(K_0,K_1,\dots)$ of control laws $K_t:(x_t,\dots,x_0) \mapsto u_t$. The closed loop of $\bm{K}$ and \eqref{eq:syss} is then described by the equation
\begin{align}
        \label{eq:syss-2} x_{t+1} &= A_0 x_t + K_t(x_t,\dots,x_0) + w_t.
\end{align}
Our goal is now to design $\bm{K}$ such that the closed-loop \eqref{eq:syss-2} is bounded-input bounded-output stable for any $A_0$ and any bounded sequence\footnote{We use bracket notation to distinguish a sequence $(w_t)$ from its element $w_t$ at time $t$} $(w_t)$. More specifically, we want to design $\bm{K}$ such that any closed-loop trajectory $(x_t)$ satisfies bounds of the form
\begin{align*}
        \sup_t \|x_t\| &\leq f_1(A_0, \sup_t\|w_t\|) \\
        \limsup_{t \rightarrow \infty} \|x_t\| &\leq f_2(A_0, \sup_t\|w_t\|)
\end{align*}
for some fixed functions $f_1$ and $f_2$.

\section{Preview of Main Result}\label{sec:previewresult}
\noindent We will start by describing the implementation of our proposed controller and a summary of its performance guarantees in a closed-loop with system \eqref{eq:syss}.

\subsection{Proposed control strategy}



\noindent For adaptive stabilization of \eqref{eq:syss} we propose a dynamic controller $\CC=(\CCt{0},\CCt{1},\dots)$, which at every time step $t$ computes the input as $$u_t=\CCt{t}(x_t,\dots,x_0, u_{t-1},\dots,u_0)$$ based on all previous measurements $x_t,\dots,x_0$ and previously taken actions $u_{t-1},\dots,u_0$. The controller $\CC$ computes the input $u_t$ at time $t$ as
\begin{align}\label{eq:previewcc}
        \CCt{t}(x_t,\dots,x_0, u_{t-1},\dots,u_0):=\left(U_{t-1} -X^+_{t-1} \right)\lambda_{t-1} (x_t)
\end{align}
where $\lambda_{t-1}(x_t)$ is defined as the solution of the convex optimization problem
\begin{align}\label{eq:previewopt}
        \begin{array}{rl}
        \min\limits_{\lambda} & \left\| \lambda\right\|_{1} \\
        \mathrm{s.t.}& X_{t-1} \lambda = x_t 
        \end{array}
\end{align}
and where the matrices $U_t$, $X_t$ and $X^+_t$ are composed of state $x_t$ and input $u_t$ measurements up until time $t$ as 
\begin{subequations}\label{subeq:matrix}
        \begin{align}
                X_{t} &:= [x_t,x_{t-1},\dots,x_0,X_{-1}] \\
                U_{t} &:= [u_t,u_{t-1},\dots,u_0,U_{-1}] \\
                X^+_{t-1} &:= [x_t,x_{t-1},\dots,x_1,X^+_{-1}]
        \end{align}      
\end{subequations}

\noindent with fixed chosen matrices $X_{-1}$, $U_{-1}$, $X^+_{-1} \in \R^{n \times n_0}$ such that $n_0 > n$ and $\mathrm{rank}(X_{-1})=n$.
The matrices $X_{-1}$, $U_{-1}$, $X^+_{-1}$ with columns $\hat{x}_i$, $\hat{u}_i$ and $\hat{x}^+_i$ defined as 
\begin{align}\label{eq:initXU}
        \begin{array}{c}
                X_{-1} := [\hat{x}_1,\dots,\hat{x}_{n_0}],\quad  
                U_{-1} :=  [\hat{u}_1,\dots,\hat{u}_{n_0}],\\ 
                X^+_{-1} := [\hat{x}^+_1,\dots,\hat{x}^+_{n_0}]
        \end{array}
\end{align} 
 serve to initialize the controller $\CC$. Depending on the application scenario, the matrices can be chosen as follows:
 \begin{enumerate}[label=(I\arabic*)]
         \item\label{it:init1} \textbf{no prior knowledge:} choose 
         $\hat{u}_i=0$, $\hat{x}^+_i=0$, $\hat{x}_i = \eps e_i$ for $1 \leq i \leq n$ where $e_i$ is the $i$th cartesian unit vector and $\eps>0$ some positive scalar. The parameter $\eps$ is can be viewed as an initial guess on $\sup_t \|w_t\|_1$, the supremum of the disturbance sequence in $1$-norm.
         \item\label{it:init2} \textbf{prior data available:} Assume we had noisy data available $x_{-j}$, $x^{+}_{-j}$, $u_{-j}$, $1\leq j \leq k$ collected from the system \eqref{eq:syss} before $t=0$. I.e. the data satisfies
         \begin{align}
                x^{+}_{-j} = A_0 x_{-j} + u_{-j} + w_{-j}
         \end{align}
         with $w_{-j}$ denoting the corresponding lumped disturbances. Then, in addition to the \ref{it:init1}-initialization $\hat{u}_i=0$, $\hat{x}^+_i=0$, $\hat{x}_i = \eps e_i$ for $1 \leq i \leq n$, we can incorporate the data $x_{-j}$, $x^{+}_{-j}$, $u_{-j}$ by appending additional columns as $$\hat{u}_{n+i} := u_{-i},\;\;\hat{x}_{n+i} := x_{-i},\;\;\hat{x}^+_{n+i} := x^+_{-i},\; 1\leq i \leq k$$
 \end{enumerate}
 

\subsection{ Closed loop guarantees }
\noindent In this paper, we will show that the controller $\CC$
stabilizes system \eqref{eq:sys} without requiring any knowledge of $A_0$ or $(w_t)$. We will term the controller $\CC$ defined by \eqref{eq:previewcc}, \eqref{eq:previewopt} the \textit{causal cancellation controller}, as it can be interpreted at time $t$ to cancel out the part of the dynamics that can be inferred from all previously collected observations $x_{t},\dots, x_{0}$ and actions $u_{t-1},\dots,u_{0}$.\\

\noindent As presented in detail in \secref{sec:mainresult}, for \textbf{any} initialization $X_{-1}$, $U_{-1}$, $X^+_{-1}$, (only assuming $\mathrm{rank}(X_{-1})=n$) the controller \eqref{eq:previewopt} \textbf{always} ensures a closed loop for which:
\begin{enumerate}[label=(\roman*)]
        \item the state $(x_t)$ and input $(u_t)$ are uniformly bounded,
        \item an analytic upper-bound can be derived for the worst-case state-deviation, 
        \item after some finite time, $(x_t)$ and $(u_t)$ converge exponentially to a bounded limit set.
\end{enumerate}
The above guarantees will be phrased w.r.t. a norm $\Nxy{\anon}{\s{W}}$ which measures $(x_t)$ and $(u_t)$ relative to the size of the disturbance $(w_t)$ that produced them. 

\noindent Moreover, as described in \ref{it:init2}, we can incorporate prior data into the initialization of the controller $\CC$. In the case where the provided data is "more informative" than the default initialization \ref{it:init1}, the closed loop guarantees and bounds tighten. Hence, the proposed control scheme $\CC$ does not need prior knowledge to give closed-loop stability guarantees, but if prior knowledge is available, it can be leveraged through the initialization \ref{it:init2} to improve closed-loop guarantees.

\section{The model-free adaptive controller\\ and closed-loop equations}\label{sec:mfadap}
\noindent We will start by first discussing the key idea and intuition behind the causal cancellation controller $\CC$.

\secref{subsec:dataequation} is deriving that all $(X^+_{t}, X_{t}, U_{t})$ satisfies the open loop equation of the unknown system for some appropriately defined disturbance matrix $W_{t}$. This is used to show in \secref{subsec:mfapprox} that at time $t$ and state $x_t$, the causal cancellation control law $\CCt{t}$ approximates the ideal deadbeat control action $u^*_t = -A_0x_t$ directly from online data $(X^+_{t-1}, X_{t-1}, U_{t-1})$ without requiring to explicitly estimate $A_0$. This relation leads to a model-free form of the closed loop equation, shown in \secref{subsec:mfcl}, which is used for the later stability analysis. 


\subsection{Open loop equation for data matrices}\label{subsec:dataequation}
\noindent Recall from \eqref{subeq:matrix}, that $(X^+_{t}, X_{t}, U_{t})$ are constructed from some fixed initialization $(X^+_{-1}, X_{-1}, U_{-1})$ and some state $(x_t)$ and input $(u_t)$ sequences of the system \eqref{eq:syss} with respect to some fixed lumped disturbance $(w_t)$. Define the disturbance matrix $W_t \in \R^{n \times (t+1+n_0)}$ as the matrix 
\begin{subequations}
        \label{eq:Wdef}
\begin{align}
        \label{eq:Wmat}W_{t} &:= [w_t,w_{t-1},\dots,w_0,W_{-1}]\\
        \label{eq:Winit}W_{-1} & := [\hat{w}_1,\dots,\hat{w}_{n_0}]:= X^+_{-1}-A_0 X_{-1}-U_{-1}
\end{align}
\end{subequations}
of lumped disturbances $w_t,\dots, w_0$ and the matrix $W_{-1}$ which is composed of the columns $\hat{w}_1,\dots,\hat{w}_{n_0}$. With the above auxiliary definition we can easily see that the matrices $U_t$, $X_t$, $X^+_t$ and $W_t$ are satisfying the linear equation
\begin{align}\label{eq:mateq}
X^{+}_{t-1} = A_0 X_{t-1} + U_{t-1} + W_{t-1},
\end{align}
which resembles the open-loop dynamics of the unknown system.

We will term $\hat{w}_i$ to be "virtual disturbances", which are defined to account for errors introduced through the initial guesses $X_{-1}$, $U_{-1}$, $X^{+}_{-1}$. If we take $\hat{x}_i, \hat{u}_i, \hat{x}^+_i$ to be the $i$th columns of the initialization matrices $X_{-1}$, $U_{-1}$, $X^{+}_{-1}$ and $W_{-1}$, we can rewrite the definition \eqref{eq:Winit} column-wise in the form
\begin{align}
        \hat{x}^+_i &= A_0 \hat{x}_i + \hat{u}_i + \hat{w}_i,\quad 1 \leq i \leq n_0.
\end{align}
to see that each pair of ($\hat{x}_i, \hat{u}_i$) and $\hat{x}^+_i$ can be posed as a transition of the 
 true unknown system \eqref{eq:syss}, w.r.t to the virtual disturbance $\hat{w}_i$. 
 \begin{ex}
        If we initialize according to procedure \ref{it:init1}, then $\hat{w}_i = -\eps A_0 e_i$ and $W_{-1} = -\eps A_0$.
\end{ex}

\subsection{Model free approximation of deadbeat control action}\label{subsec:mfapprox}
\noindent In a compact form, the causal cancellation control laws $\CCt{t}: (x_t,X^+_{t-1}, X_{t-1}, U_{t-1})\mapsto u_t$ are defined as
\begin{align} \label{eq:cc2}
        &u_t = \left(U_{t-1} -X^+_{t-1} \right)\lambda_{t-1} (x_t)\\
         \notag &\text{where}\;\;\lambda_{t-1} (x_t) := \argmin\limits_{\lambda \text{ s.t. }X_{t-1}\lambda=x_t} \Nxy{\lambda}{1}
\end{align}
\begin{rem*}
        The technical issue that a minimizer of \eqref{eq:previewopt} might not be unique is not relevant for the analysis and for simplicity will be ignored. 
\end{rem*}


\noindent The function $\lambda_{t-1}(x_t)$ is defined to always satisfy
 \begin{align}\label{eq:rep}
X_{t-1}\lambda_{t-1}(x_t) = x_t.
\end{align}
 and represents a decomposition of the state $x_t$ as a linear combination of the columns of $X_{t-1}$.

\noindent Rewriting equation \eqref{eq:mateq} as 
\begin{align}\label{eq:diffmateq}
        U_{t-1}-X^{+}_{t-1} = -A_0 X_{t-1}-W_{t-1} 
\end{align}
and substituting the right hand side of equation \eqref{eq:diffmateq} into \eqref{eq:cc2} and using \eqref{eq:rep} allows us to rewrite the controller equivalently as 
\begin{align}
        \label{eq:rewrite} u_t &= \left(U_{t-1} -X^+_{t-1} \right)\lambda_{t-1} (x_t)\overset{\eqref{eq:rep}}{=}-A_0 x_{t} - W_{t-1}\lambda_{t-1}(x_{t})
\end{align}
The above shows that the control law \eqref{eq:cc2} is a direct way to approximate the ideal deadbeat control action $-A_0x_t$ from the online data matrices $U_{t-1}$, $X^+_{t-1}$, $X_{t-1}$. The additional term $- W_{t-1}\lambda_{t-1}(x_{t})$ is the corresponding approximation error at time $t$. As will become more clear later, the optimization step in \eqref{eq:cc2} is minimizing an upperbound of $- W_{t-1}\lambda_{t-1}(x_{t})$ w.r.t. to the norm $\Nxy{\anon}{\s{W}_{t-1}}$ (See definition in \secref{subsec:notdef}).
\subsection{The model-free closed-loop equations}\label{subsec:mfcl}
\noindent Setting $\bm{K}=\CC$ and using \eqref{eq:rewrite}, the closed loop equations \eqref{eq:syss-2} take the form of
\begin{subequations}
        \begin{align}
                \label{eq:clloop} x_{t+1} &=-W_{t-1}\lambda_{t-1}(x_t)+w_t.\\
                u_t &= \left(U_{t-1} -X^+_{t-1} \right)\lambda_{t-1} (x_t)
        \end{align}     
\end{subequations}

\noindent which will serve as the basis for our stability analysis. The above equation says that the closed loop dynamics are entirely determined by $W_{t-1}$ (containing virtual and past lumped disturbances), as well as how we choose to decompose $x_t$ as a linear combination $X_{t-1}\lambda_{t-1}(x_t)$ of past data. Suitably, we could call equation \eqref{eq:clloop} to be a model-free description of the closed loop, since the dynamics are formulated independent of the underlying true unknown system $A_0$. 

\section{Main results}\label{sec:mainresult}

\noindent \thmref{thm:finiteunstable} and \thmref{thm:main} are the main theoretical results. \thmref{thm:finiteunstable} states that any state trajectory $(x_t)$ of the closed loop has finitely many "unstable transitions" (defined in \defref{def:unstabletransition}). 
\thmref{thm:main} is a consequence of \thmref{thm:finiteunstable} and presents our main stability bounds for state and input trajectories of the closed loop.

To formulate our results, we first introduce necessary notation and definitions.
\subsection{Notation and Definitions}\label{subsec:notdef}
\begin{defn}\label{def:matrixsetnotation}
If $M\in\R^{n\times N}$ is a matrix with $N$ columns $m_i$, then define the corresponding variable $\s{M}$ in sans serif font to denote the set $\s{M}:=\{m_1,\dots, m_N\}$.
\end{defn}
\begin{defn}\label{def:abscvx-0}
        Let $\s{S}$ be a set in $\R^n$, then the set of all finite linear combinations $\sum^{N}_{i=1} \lambda_i x_i$ of elements $x_i$ in $\s{S}$ with $\sum^{N}_{i=1} |\lambda_i| \leq 1$ is called the \textit{absolute convex hull} of $\s{S}$ and we will refer to its closure as $\cx{\s{S}}$:
        \begin{align}\label{eq:defcxS}
        \cx{\s{S}} := \mathrm{cl}\left\{\sum^{N}_{i=1} \lambda_i x_i\left|\quad \{x_i\}^{N}_{i=1} \subset \s{S}, \quad \sum^{N}_{i=1} |\lambda_i| \leq 1 \right.\right\}.
        \end{align}
\end{defn}
\begin{defn}\label{def:repnorm-0}
        For a fixed bounded set $\s{S} \subset \R^n$, let $\Nxy{\anon}{\s{S}}:\R^n \mapsto \R_{\geq 0}$ be the norm defined for all $x\in\R^n$ as
\begin{align}\label{eq:defrepnorm}
        \Nxy{x}{\s{S}} := \left\{\begin{array}{ll} \min \left\{r \geq 0 \left| x \in r \cx{\s{S}}\right.\right\}, & \text{for }x\in \mathrm{span}(\s{S}) \\ \infty, &\text{else}
         \end{array}\right.
\end{align}
and for sets $\s{S}' \subset \R^n$, define $\Nxy{\s{S}'}{\s{S}}$ as the quantity
\begin{align}\label{eq:s1s2}
        \Nxy{\s{S}'}{\s{S}} &:= \max\limits_{z \in \cx{\s{S}'}} \Nxy{z}{\s{S}}
\end{align}
\end{defn}
\noindent Key properties of the above norm and relevant concepts from convex geometry are discussed in the appendix \secref{sec:appcvx}. 
\noindent For a fixed disturbance $(w_t)$ and virtual disturbance $\hat{w}^i$, define $\s{W}$ as the corresponding fixed set 
\begin{align}\label{eq:Wdef}
        \s{W} := \{w_t|t \in \mathbb{N}\}\cup\{\hat{w}_i| 1 \leq i \leq n_0 \}.
\end{align} 

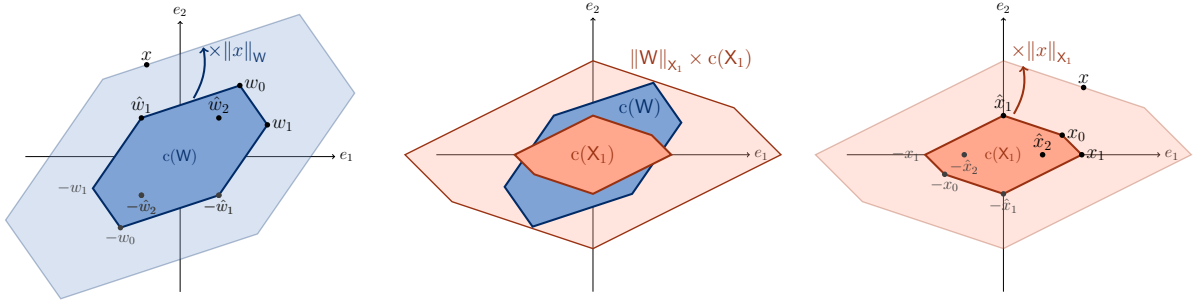
\begin{figure*}
        \centering
        \resizebox{0.3\textwidth}{!}{
        \begin{tikzpicture}[font=\large] 

      \tikzset{invclip/.style={clip,insert path={{[reset cm]
      (-16383.99999pt,-16383.99999pt) rectangle (16383.99999pt,16383.99999pt)}}}}
      \definecolor{K}{HTML}{0047AB}
      \definecolor{Kcopy}{HTML}{66B2DD}
      \definecolor{blowup}{HTML}{FF511C}
      \definecolor{cone}{HTML}{FFBE00}
      \definecolor{ballpair}{HTML}{00C855}
      \colorlet{ball}{blowup!50!cone}
  
      \def \light{ 50 }
      \def \dark{ 60 }
      \def \lthick{ 1.5pt }
      \def \lthin{ 1pt }
      \def \tta{ -10.00000000000000 } 
      \def \k{    -3.00000000000000 } 
      \def \l{     6.00000000000000 } 
      \def \d{     5.00000000000000 } 
      \def \h{     7.00000000000000 } 
      \def \r{     1.0000 } 
      \def \lam{   2.0 } 
      \def \xlen{   8.0 }
      \def \ylen{   7.0 }
      \def \pang{    70.0 } 
      \def \prad{   1.0}
      \def \del{   1.6 } 
      \pgfmathparse{1/\del}
                  \let\invdel\pgfmathresult
  
      \def \beta{ 20.0}
      
      \clip (-5,-4) rectangle (5,4);
      \coordinate (A) at (-\r,-\r); 
      \coordinate (B) at (\r,-\r); 
      \coordinate (C) at (\r,\r); 
      \coordinate (D) at (-\r,\r); 
      \coordinate (O) at (0,0);
      \coordinate (xleft) at ({-0.5*\xlen},0);
      \coordinate (xright) at ({0.5*\xlen},0);
      \coordinate (ydown) at (0,{-0.5*\ylen});
      \coordinate (yup) at (0,{0.5*\ylen});
  \def \dela{ 1.6}
  \def \rpos{2.4}
  \def \rnorm{2}
  \pgfmathparse{sqrt(\dela)}\let\sqrtdel\pgfmathresult
  \pgfmathparse{sqrt(\dela)*\rpos}\let\normdel\pgfmathresult
  \def \anga{ 160.0}
  \def \angb{ 130}
  \def \angc{ 90}
  \def \angd{ 50}
  \def \ange{ 20}

  \coordinate (W1) at (-\r,\r); 
  \coordinate (W2) at (\r,\r); 
  \coordinate (W3) at (\angd:\rpos);
  \coordinate (W4) at (\ange:\rpos);

  \coordinate (PX) at (intersection of O--(110:1) and W1--W3);

  \coordinate (PY) at (intersection of O--(80:1) and W1--W3);

  \newcommand{\anno}[3]{
  \begin{scope}
      \coordinate (X) at #1;
      \fill[black]    (X) circle [radius=2pt];  
      \draw (X) node [#3,black]     {\Large #2};
  \end{scope}
  }
  \newcommand{\annob}[3]{
      \begin{scope}
          \coordinate (X) at #1;
          \fill[gray!50!black]    (X) circle [radius=2pt];  
          \draw (X) node [#3,gray!50!black]     {\large #2};
      \end{scope}
 }
  
  \draw[->] (xleft)--(xright) node[right]{$e_1$};
  \draw[->] (ydown)--(yup) node[above]{$e_2$};

  \filldraw[line width = \lthin, fill = K!\light, draw=K!\dark!black, opacity = 0.3] 
  ($\rnorm*(W1)$) -- ($\rnorm*(W3)$) -- ($\rnorm*(W4)$) -- ($-\rnorm*(W1)$) -- ($-\rnorm*(W3)$) -- ($-\rnorm*(W4)$) -- cycle;
  \filldraw[line width = \lthick, fill = K!\light, draw=K!\dark!black] 
  (W1) -- (W3) -- (W4) -- ($-1*(W1)$) -- ($-1*(W3)$) -- ($-1*(W4)$) -- cycle;
\anno{(W1)}{$\hat{w}_1$}{above}
\anno{(W2)}{$\hat{w}_2$}{above}
\annob{($-1*(W1)$)}{$-\hat{w}_1$}{below}
\annob{($-1*(W2)$)}{$-\hat{w}_2$}{below}
\anno{(W3)}{$w_0$}{right}
\anno{(W4)}{$w_1$}{right}
\annob{($-1*(W4)$}{$-w_1$}{left}
\annob{($-1*(W1)$)}{$-\hat{w}_1$}{below}
\annob{($-1*(W2)$)}{$-\hat{w}_2$}{below}
\annob{($-1*(W3)$)}{$-w_0$}{below}

\anno{($\rnorm*(PX)$)}{$x$}{above}

\draw [->, K!\dark!black, line width = \lthick] ($(PY)+(0.1,0.1)$.east) 
to [out=60, in=-80] ($\rnorm*(PY)+(0.1, 0.0)$.east);

\coordinate (labelarrow) at ($0.5*(PY) + 0.5*\rnorm*(PY) + (1.1,0.6)$);
\draw (labelarrow) node [line width = \lthick, K!\dark!black]     {\Large $\times \Nxy{x}{\s{W}}$};

\draw (O) node [K!\dark!black]     {\large $\cx{\s{W}}$};

  \end{tikzpicture}
        }     
        \resizebox{0.3\textwidth}{!}{
                \begin{tikzpicture}[font=\large] 

      \tikzset{invclip/.style={clip,insert path={{[reset cm]
      (-16383.99999pt,-16383.99999pt) rectangle (16383.99999pt,16383.99999pt)}}}}
      \definecolor{K}{HTML}{0047AB}
      \definecolor{Kcopy}{HTML}{66B2DD}
      \definecolor{blowup}{HTML}{FF511C}
      \definecolor{cone}{HTML}{FFBE00}
      \definecolor{ballpair}{HTML}{00C855}
      \colorlet{ball}{blowup!50!cone}
  
      \def \light{ 50 }
      \def \dark{ 60 }
      \def \lthick{ 1.5pt }
      \def \lthin{ 1pt }
      \def \tta{ -10.00000000000000 } 
      \def \k{    -3.00000000000000 } 
      \def \l{     6.00000000000000 } 
      \def \d{     5.00000000000000 } 
      \def \h{     7.00000000000000 } 
      \def \r{     1.0000 } 
      \def \lam{   2.0 } 
      \def \xlen{   8.0 }
      \def \ylen{   7.0 }
      \def \pang{    70.0 } 
      \def \prad{   1.0}
      \def \del{   1.6 } 
      \pgfmathparse{1/\del}
                  \let\invdel\pgfmathresult
  
      \def \beta{ 20.0}
      
      \clip (-5,-4) rectangle (5,4);
      \coordinate (A) at (-\r,-\r); 
      \coordinate (B) at (\r,-\r); 
      \coordinate (C) at (\r,\r); 
      \coordinate (D) at (-\r,\r); 
      \coordinate (O) at (0,0);
      \coordinate (xleft) at ({-0.5*\xlen},0);
      \coordinate (xright) at ({0.5*\xlen},0);
      \coordinate (ydown) at (0,{-0.5*\ylen});
      \coordinate (yup) at (0,{0.5*\ylen});

  \def \dela{ 1.6}
  \def \rpos{2.4}
  \def \rnorm{2.4}
  \pgfmathparse{sqrt(\dela)}\let\sqrtdel\pgfmathresult
  \pgfmathparse{sqrt(\dela)*\rpos}\let\normdel\pgfmathresult
  \def \anga{ 160.0}
  \def \angb{ 130}
  \def \angc{ 90}
  \def \angd{ 50}
  \def \ange{ 20}

  \coordinate (X1) at (0,\r); 
  \coordinate (X2) at (\r,0); 
  \coordinate (X3) at ({1.5*\r}, {0.5*\r});
  \coordinate (X4) at ({2*\r},0);
  \coordinate (Xn4) at ({-2*\r},0);

  \coordinate (W1) at (-\r,\r); 
  \coordinate (W2) at (\r,\r); 
  \coordinate (W3) at (\angd:\rpos);
  \coordinate (W4) at (\ange:\rpos);

  \coordinate (PX) at (intersection of O--(110:1) and Xn4--X1);

  \coordinate (PY) at (intersection of O--(80:1) and X1--X3);

  \newcommand{\anno}[3]{
  \begin{scope}
      \coordinate (X) at #1;
      \fill[black]    (X) circle [radius=2pt];  
      \draw (X) node [#3,black]     {\Large #2};
  \end{scope}
  }
  \newcommand{\annob}[3]{
      \begin{scope}
          \coordinate (X) at #1;
          \fill[gray!50!black]    (X) circle [radius=2pt];  
          \draw (X) node [#3,gray!50!black]     {\large #2};
      \end{scope}
 }
  
  \draw[->] (xleft)--(xright) node[right]{$e_1$};
  \draw[->] (ydown)--(yup) node[above]{$e_2$};


  \filldraw[line width = \lthin, fill = blowup!\light, draw=blowup!\dark!black, fill opacity = 0.3] 
  ($\rnorm*(X1)$) -- ($\rnorm*(X3)$) -- ($\rnorm*(X4)$) -- ($-\rnorm*(X1)$) -- ($-\rnorm*(X3)$) -- ($-\rnorm*(X4)$) -- cycle;
  \filldraw[line width = \lthick, fill = K!\light, draw=K!\dark!black] 
  (W1) -- (W3) -- (W4) -- ($-1*(W1)$) -- ($-1*(W3)$) -- ($-1*(W4)$) -- cycle;

  \filldraw[line width = \lthick, fill = blowup!\light, draw=blowup!\dark!black] 
  (X1) -- (X3) -- (X4) -- ($-1*(X1)$) -- ($-1*(X3)$) -- ($-1*(X4)$) -- cycle;



\coordinate (labelarrow) at (2.5, 2.4);
\draw (labelarrow) node [line width = \lthick, blowup!\dark!black]     {\Large $\Nxy{\s{W}}{\s{X}_1} \times \cx{\s{X}_1}$ };

\draw (O) node [blowup!\dark!black]     {\Large $\cx{\s{X}_1}$};
\coordinate (O1) at (1.2,1.2);
\draw (O1) node [K!\dark!black]     {\Large $\cx{\s{W}}$};

  \end{tikzpicture}}   
        \resizebox{0.3\textwidth}{!}{
        \begin{tikzpicture}[font=\large] 

      \tikzset{invclip/.style={clip,insert path={{[reset cm]
      (-16383.99999pt,-16383.99999pt) rectangle (16383.99999pt,16383.99999pt)}}}}
      \definecolor{K}{HTML}{FF511C}
      \definecolor{Kcopy}{HTML}{66B2DD}
      \definecolor{blowup}{HTML}{FF511C}
      \definecolor{cone}{HTML}{FFBE00}
      \definecolor{ballpair}{HTML}{00C855}
      \colorlet{ball}{blowup!50!cone}
  
      \def \light{ 50 }
      \def \dark{ 60 }
      \def \lthick{ 1.5pt }
      \def \lthin{ 1pt }
      \def \tta{ -10.00000000000000 } 
      \def \k{    -3.00000000000000 } 
      \def \l{     6.00000000000000 } 
      \def \d{     5.00000000000000 } 
      \def \h{     7.00000000000000 } 
      \def \r{     1.0000 } 
      \def \lam{   2.0 } 
      \def \xlen{   8.0 }
      \def \ylen{   7.0 }
      \def \pang{    70.0 } 
      \def \prad{   1.0}
      \def \del{   1.6 } 
      \pgfmathparse{1/\del}
                  \let\invdel\pgfmathresult
  
      \def \beta{ 20.0}
      
      \clip (-5,-4) rectangle (5,4);
      \coordinate (A) at (-\r,-\r); 
      \coordinate (B) at (\r,-\r); 
      \coordinate (C) at (\r,\r); 
      \coordinate (D) at (-\r,\r); 
      \coordinate (O) at (0,0);
      \coordinate (xleft) at ({-0.5*\xlen},0);
      \coordinate (xright) at ({0.5*\xlen},0);
      \coordinate (ydown) at (0,{-0.5*\ylen});
      \coordinate (yup) at (0,{0.5*\ylen});
  \def \dela{ 1.6}
  \def \rpos{2.4}
  \def \rnorm{2.4}
  \pgfmathparse{sqrt(\dela)}\let\sqrtdel\pgfmathresult
  \pgfmathparse{sqrt(\dela)*\rpos}\let\normdel\pgfmathresult
  \def \anga{ 160.0}
  \def \angb{ 130}
  \def \angc{ 90}
  \def \angd{ 50}
  \def \ange{ 20}

  \coordinate (W1) at (0,\r); 
  \coordinate (W2) at (\r,0); 
  \coordinate (W3) at ({1.5*\r}, {0.5*\r});
  \coordinate (W4) at ({2*\r},0);
  \coordinate (Wn4) at ({-2*\r},0);

  \coordinate (PX) at (intersection of O--(40:1) and W1--W3);

  \coordinate (PY) at (intersection of O--(80:1) and W1--W3);

  \newcommand{\anno}[3]{
  \begin{scope}
      \coordinate (X) at #1;
      \fill[black]    (X) circle [radius=2pt];  
      \draw (X) node [#3,black]     {\Large #2};
  \end{scope}
  }
  \newcommand{\annob}[3]{
      \begin{scope}
          \coordinate (X) at #1;
          \fill[gray!50!black]    (X) circle [radius=2pt];  
          \draw (X) node [#3,gray!50!black]     {\large #2};
      \end{scope}
 }
  
  \draw[->] (xleft)--(xright) node[right]{$e_1$};
  \draw[->] (ydown)--(yup) node[above]{$e_2$};

  \filldraw[line width = \lthin, fill = K!\light, draw=K!\dark!black, opacity = 0.3] 
  ($\rnorm*(W1)$) -- ($\rnorm*(W3)$) -- ($\rnorm*(W4)$) -- ($-\rnorm*(W1)$) -- ($-\rnorm*(W3)$) -- ($-\rnorm*(W4)$) -- cycle;
  \filldraw[line width = \lthick, fill = K!\light, draw=K!\dark!black] 
  (W1) -- (W3) -- (W4) -- ($-1*(W1)$) -- ($-1*(W3)$) -- ($-1*(W4)$) -- cycle;
\anno{(W1)}{$\hat{x}_1$}{above}
\anno{(W2)}{$\hat{x}_2$}{above}
\annob{($-1*(W1)$)}{$-\hat{x}_1$}{below}
\annob{($-1*(W2)$)}{$-\hat{x}_2$}{below}
\anno{(W3)}{$x_0$}{right}
\anno{(W4)}{$x_1$}{right}
\annob{($-1*(W4)$}{$-x_1$}{left}
\annob{($-1*(W3)$)}{$-x_0$}{below}

\anno{($\rnorm*(PX)$)}{$x$}{above}

\draw [->, K!\dark!black, line width = \lthick] ($(PY)+(0.1,0.1)$.east) 
to [out=60, in=-80] ($\rnorm*(PY)+(0.1, 0.0)$.east);

\coordinate (labelarrow) at ($0.5*(PY) + 0.5*\rnorm*(PY) + (0.7,1.0)$);
\draw (labelarrow) node [line width = \lthick, K!\dark!black]     {\Large $\times \Nxy{x}{\s{X}_1}$};

\draw (O) node [K!\dark!black]     {\large $\cx{\s{X}_1}$};

  \end{tikzpicture}}

        \caption{Examples in $\R^2$. \textit{Left}: $\cx{\s{W}}$ and $\Nxy{x}{\s{W}}$ for $\s{W} = \{\hat{w}_1,\hat{w}_2\} \cup \{w_t| r \in \mathbb{N}\}$ and $(w_t) = (w_0,w_1,0,0,\dots)$. \textit{Right}: $\cx{\s{X}_1}$ and $\Nxy{\anon}{\s{X}_1}$ for $X_1 = [x_1,x_0,\hat{x}_2,\hat{x}_1]$. \textit{Middle}: $\kappa_2 = \Nxy{\s{W}}{\s{X}_1}$ is the smallest factor $r$ such that $\cx{\s{W}} \subset r \cx{\s{X}_1}$. }
        \label{fig:norm1}
\end{figure*}
\noindent Let $\Nxy{\cdot}{\s{W}}$ and $\Nxy{\cdot}{\s{X}_{t}}$ denote the norms constructed from the fixed disturbances and data matrix $X_{t}$ according to \defref{def:repnorm-0} and \defref{def:matrixsetnotation}. For a fixed trajectory $(x_t)$, $\s{W}$ and fixed initial time $\tau$, the constant $\kk$ refers to the quantity
\begin{align}
\kk := \BXtnull.
\end{align}
\figref{fig:norm1} shows an example in $\R^2$ that illustrates the geometric relationship between the sets $\cx{\s{W}}$, $\cx{\s{X_t}}$ and the evaluation of their respective norms at some point $x$. The arrows indicate that one set is a scaled copy of the other set. The middle picture in \figref{fig:norm1} shows a geometric interpretation of the corresponding constant $\kk$ for $\tau=2$ : $\kk$ is the smallest scaling factor $r$ such that the set $r \times \cx{\s{X}_{\tnull-1}}$ contains the set $\cx{\s{W}}$. 

\subsection{Finite occurrence of unstable transitions}\label{sec:finiteocc}

\noindent Our approach is to analyze the behavior of the closed loop by quantifying how many "unstable transitions" can occur in the future time window $[\tnull,\infty)$ of a closed loop trajectory $(x_t)$, given ($X_{\tnull}$, $X^+_{\tnull-1}$, $U_{\tnull-1}$), which represents the data collected up until time $\tau$. For a fixed $0<\mu <1$ and a trajectory $(x_t)$, we define the occurrence of a $\mu$-unstable transition as follows:

\begin{defn}[$\mu$-unstable transition] \label{def:unstabletransition}
        The trajectory $(x_t)$ has a $\mu$-unstable transition at time $t$ if the pair of consecutive states $(x_{t+1},x_t)$ satisfies
        \begin{align}\label{eq:unstabdef}
        \Nxy{x_{t+1}}{\s{W}} > \max \left\{ \tfrac{1}{1-\mu}, \mu \Nxy{x_t}{\s{W}} +1  \right\}.
        \end{align}
        or in short $(x_{t+1},x_{t}) \in \c{U}_\mu$, where $\c{U}_\mu$ denotes the set of all pairs $(x,x^+) \in \R^n\times \R^n$ that satisfy the inequality \eqref{eq:unstabdef}:
\begin{align}\label{eq:refunset}
        \c{U}_\mu := \Big\{(x^+,x)| \Nxy{x^+}{\s{W}} > \max \big\{ \tfrac{1}{1-\mu}, \mu \Nxy{x}{\s{W}} +1 \big\}\Big\}.
\end{align}
        \end{defn}
  \noindent The condition \eqref{eq:unstabdef} represents a growth condition on a transition $(x_t,x_{t+1})$ in the trajectory $(x_t)$. For each trajectory $(x_t)$, we define a corresponding set $\c{X}_\mu$ that collects all states $x_t$ at which $(x_{t+1},x_t)$ belongs to $\cl{U}_{\mu}$: 
  \begin{defn}\label{def:Tu}
        Given a trajectory $(x_t)$, an initial time $\tnull$ and some $0 < \mu < 1$, define $\c{X}_\mu((x_t);\tnull) \subset \R^n$ as
        \begin{align}
                \c{X}_\mu\left((x_t);\tnull\right):=\left\{\;x_t\;\left| \;\; (x_t,x_{t+1}) \in \c{U}_\mu,t\geq \tnull \right. \right\}.
        \end{align}
\end{defn}
\begin{rem*}
Note that if $\mu<\mu'$, then $\c{X}_{\mu}((x_t),\tnull) \supset \c{X}_{\mu'}((x_t),\tnull)$.
\end{rem*}


\noindent The core technical contribution of our paper is \thmref{thm:finiteunstable}, which places an upper bound on the number of $\mu$-unstable transitions that can occur in the closed loop trajectory $(x_t)$:



\begin{thm}\label{thm:finiteunstable}
For any trajectory $(x_t)$ of the closed loop \eqref{eq:clloop} and any $\tnull \geq 0$, the set $\c{X}_\mu\left((x_t);\tnull\right)$ is a finite set for any $\mu\in\c{I}_\kk$, where $\c{I}_{\kk}$ is the open interval
\begin{align}\label{eq:murange-thm-0}
        \c{I}_\kk:=\Big( \Big(\sqrt{\tfrac{1}{4}+\tfrac{1}{\kk}}+\tfrac{1}{2}\Big)^{-1} , 1\Big)
\end{align} 
Moreover, the cardinality is bounded above as $|\c{X}_\mu\left((x_t);\tnull\right)|\leq N(\mu;\kk)$, where $N:\R\times\R\mapsto\R$ stands for the function
\begin{align}\label{eq:Nbound-thm}
        {N}(\mu;\kk):=  \tfrac{1}{2}\left(\tfrac{\mu}{\mu-\sqrt{\kk (1-\mu) }}\right)^n \max \{1, \tfrac{\mu}{1-\mu} \}^n
\end{align}
and $\kk$ is a constant computed from $\s{X}_{\tnull-1}$ as:
\begin{align}\label{eq:kappa}
\kk = \BXtnull 
\end{align}
\end{thm}

\begin{rem}
        Recall, that we initialize $X_{-1}$ such that $\mathrm{rank}(X_{-1})=n$; This guarantees $\mathrm{rank}(X_{\tnull-1})=n$, assures $\kk <\infty$ and that the interval $\c{I}_\kk$ is always non-empty. 
        In addition, it can be verified that $N(\mu;\kk)<\infty$  for any feasible $\mu$.
\end{rem}
\noindent \thmref{thm:finiteunstable} states that for suitably chosen $\mu$, the set $\c{X}_\mu((x_t);\tnull)$ is finite for any closed loop trajectory $(x_t)$. 
The constant $\kk$ controls the interval of feasible $\mu$ as well as the total number of unstable transitions $\c{U}_\mu$ that can occur in the time interval $[\tnull,\infty)$. As $\kk$ decreases, the bound $N(\mu;\kk)$ tightens ($N(\mu;\kk) \leq N(\mu;\kk')$ for $\kk \leq \kk'$) and the interval \eqref{eq:murange-thm-0} widens. Geometrically, $\kk$ describes the size of the disturbance set $\s{W}$ relative to the set $\cx{\s{X}_{\tnull-1}}$ (see \figref{fig:norm1} for an example in $\R^2$) and the result states that we have less unstable transitions if the observations collected are larger in size than the disturbance. We can therefore view $\kk$ as a constant which quantifies how informative is the data $\s{X}_{\tnull-1}$ observed before $\tnull$ to control the system for time $t\geq \tnull$. \\
The proof of \thmref{thm:finiteunstable} is postponed for \secref{sec:boundingproof}. 

\subsection{Closed loop stability bounds}\label{sec:clbounds}
\noindent As a consequence of \thmref{thm:finiteunstable}, we obtain our main closed loop stability bounds presented in \thmref{thm:main}.
The result gives bounds on the trajectories $(x_t)$ and $(u_t)$ in terms of the fixed disturbance $(w_t)$ and virtual disturbance $\hat{w}_i$. 
\begin{rem*}
Recall from \eqref{eq:Winit} that instead of analyzing the closed loop dynamics for fixed $A_0$, $X^+_{-1}$, $X_{-1}$, $U_{-1}$, we can equivalently analyze the closed loop dynamics for fixed $\hat{w}_i$.

\end{rem*}
\begin{thm}\label{thm:main}
        Let $(x_t)$, $(u_t)$ be trajectories of the closed loop \eqref{eq:clloop} for some fixed $(w_t)$ and $\hat{w}_i$ with the corresponding set $\s{W}$ defined as \eqref{eq:Wdef}.
        \noindent Let $\tnull$ be some fixed time and let $\kk:=\BXtnull$. Then, for any $\mu \in \c{I}_{\kk}$, where $\c{I}_{\kk}$ is the interval
        \begin{align}\label{eq:murange-thm}
                \c{I}_\kk:=\Big( \Big(\sqrt{\tfrac{1}{4}+\tfrac{1}{\kk}}+\tfrac{1}{2}\Big)^{-1} , 1\Big),
        \end{align} 
        
        \noindent the trajectories $(x_t)$ and $(u_t)$ satisfy the bounds \ref{it:limsup}, \ref{it:exp} and \ref{it:supbound},
        \begin{enumerate}[label=(\roman*)]
                \item \label{it:limsup} $\limsup_{t\rightarrow \infty} \Nxy{x_t}{\s{W}} \leq \tfrac{1}{1-\mu}$,
                
                $\limsup_{t\rightarrow \infty} \Nxy{u_t}{\s{W}} \leq (\|A_0\|_{\s{W}} + \kk) \tfrac{1}{1-\mu}$
                \item \label{it:exp} there exists an $T'>0$ such that for all $k>0$ holds
                \begin{align}
                        V_1(x_{T'+k})  \leq \mu^{k} V_1(x_{T'})
                \end{align}
                where $V_1(x) := \max\{0,\Nxy{x}{\s{W}}-\tfrac{1}{1-\mu}\}$.
                \item \label{it:supbound} 
                the worst-case norm of $(x_t)$ and $(u_t)$ is bounded above as \footnote{for $\tau=0$ and $x_0 \notin \mathrm{span}(\s{W})$, replace $\Nxy{x_0}{\s{W}}$ with $\Nxy{A_0 x_0}{\s{W}}$ in \eqref{eq:xsubbound}} 
                \begin{align}
                \label{eq:xsubbound}&\sup_{t\geq \tnull} \Nxy{x_t}{\s{W}} \leq f(\kk,\mu, \Nxy{x_{\tnull}}{\s{W}}) + g(\mu,\kk)\\
                \notag  &\sup_{t\geq \tnull} \Nxy{u_t}{\s{W}} \leq (\|A_0\|_{\s{W}} + \kk)\sup_{t\geq \tnull} \Nxy{x_t}{\s{W}}
                \end{align} 
                \end{enumerate}
        \noindent where $N(\mu;\kk)$ is defined as the function 
        \begin{align}\label{eq:Nbound-thm-0}
                {N}(\mu;\kk):=  \tfrac{1}{2}\left(\tfrac{\mu}{\mu-\sqrt{\kk (1-\mu) }}\right)^n \max \{1, \tfrac{\mu}{1-\mu} \}^n,
        \end{align}
        $\|A_0\|_{\s{W}}:= \max\limits_{x \in \s{W}} \|A_0x\|_{\s{W}}$ is a constant and $f$ and $g$ abbreviate the functions
        \begin{align}
                \notag f(\kk,\mu, \Nxy{x_{\tnull}}{\s{W}}) &= \max\{1,\kk^{{N}(\mu;\kk)}\}\max\{\tfrac{1}{1-\mu}, \Nxy{x_{\tnull}}{\s{W}}\}\\
                g(\kk,\mu) &= \frac{1-\kk^{N(\mu;\kk)}}{1-\kk}.
        \end{align} 
\end{thm}
\noindent The bounds in \thmref{thm:main} are phrased w.r.t. to the norm $\Nxy{\anon}{\s{W}}$ that is constructed from the set $\s{W}$ (see \figref{fig:norm1} as an example of $\Nxy{\anon}{\s{W}}$ in $\mathbb{R}^2$). The set $\s{W}$ captures disturbances due to $(w_t)$ and due to $\hat{w}_i$, where $\hat{w}_i$ describes the mismatch between the initial guess matrices $X^{+}_{-1}$, $X_{-1}$, $U_{-1}$ and the true system matrix $A_0$. $\Nxy{x_t}{\s{W}}$ measures $x_t$ relative to the underlying set of (lumped and virtual) disturbances $\s{W}$ that realized it. \\
The result also quantifies how the bound guarantees improvement with online data:
Given some initial time $\tnull$, the above result gives stability bounds on the future trajectories of $x_t$, $u_t$, $ t\geq \tnull$ which depend on the total states observed $X_{\tnull}$ before time $\tnull$, the constant $\kk$ and $\mu \in \cl{I}_{\kk}$, which acts as a free variable. 
 The constant $\kk$ can be interpreted as a signal-to-noise ratio between state observations $X_{\tnull}$ and the disturbance set $\s{W}$ (see \figref{fig:norm1} for an example in $\R^2$). A smaller $\kk$ indicates that the data $X^{+}_{\tnull-1},X_{\tnull-1}, U_{\tnull-1}$ collected before time $\tnull$ is more informative of how to stabilize the system for future time-steps $t\geq \tau$. $\kk$ is always non-increasing in $\tnull$ and the bounds \ref{it:supbound}, \ref{it:limsup} of \thmref{thm:main} tighten as $\tnull$ increases. 
The bounds in \thmref{thm:main} depend on a free variable $\mu$ which can be chosen in the interval $\cl{I}_{\kk}$. We can tighten the bounds \ref{it:limsup} and \ref{it:supbound} by minimizing the right hand side over $\mu \in \cl{I}_{\kk}$. For bound \ref{it:limsup}, the choice 
\begin{align}
        \mu^*  =\Big(\sqrt{\tfrac{1}{4}+\tfrac{1}{\kk}}+\tfrac{1}{2}\Big)^{-1}
\end{align} 
minimizes $\tfrac{1}{1-\mu}$ over $\mu \in \cl{I}_{\kk}$ and achieves a minimal value which is almost linear in $\kk$:
\begin{align}\label{eq:mboundeq}
        \tfrac{1}{1-\mu^*} = \kk\left(\tfrac{1}{2} + \sqrt{\tfrac{1}{4}+\tfrac{1}{\kk}}\right) +1 \quad (\leq \kk +2).
\end{align}
For $\tau=0$ we get the following improved asymptotic upperbound for the state trajectory:
\begin{coro} \label{coro:initbound}If $(x_t)$ satisfies \eqref{eq:clloop} then
        $$\limsup\limits_{t\rightarrow \infty} \Nxy{x_t}{\s{W}} \leq \kappa_0\left(\tfrac{1}{2} + \sqrt{\tfrac{1}{4}+\tfrac{1}{\kappa_0}}\right) +1.$$
\end{coro}
\subsubsection{Example} 
Assume $n=1$ and the scalar system $x_{t+1} = a_0 x_t + u_t + w_t$. Pick $X_{-1}=\varepsilon$ with some $\varepsilon >0$ and $X^{+}_{-1}$, $U_{-1} = 0$. Let $(w_t)$ be some fixed bounded scalar disturbance with $\Nxy{(w_t)}{\infty} = 1$. Then $\s{W}=\cx{-a_0 \varepsilon \cup \{w_t|t \in \mathbb{N}\}}$ and $\Nxy{x}{\s{W}} = \tfrac{|x|}{\max\{|a_0|\varepsilon,1\}}$. The constant $\kappa_0$ takes the value
$$\kappa_0 = \BXtneg = \max\{\varepsilon^{-1},|a_0|\}. $$
 If we substitute this into \corref{coro:initbound}, and rewrite it in terms of $|x_t|$ we obtain the bound 
\begin{align}
        \limsup\limits_{t\rightarrow \infty} |x_t| \leq \varepsilon\kappa^2_0\left(\tfrac{1}{2} + \sqrt{\tfrac{1}{4}+\tfrac{1}{\kappa_0}}\right) +\varepsilon\kappa_0
\end{align}

\section{Proving closed loop stability}\label{sec:stabilityanalysis}
\noindent In this section, we derive the closed loop stability bounds presented in \thmref{thm:main} from the results of \thmref{thm:finiteunstable}. The derivation of \thmref{thm:finiteunstable} is postponed for \secref{sec:boundingproof}.
First, we will derive some useful inequalities which are used frequently in the derivations.
\subsection{Bounding one-time step closed loop transitions}

\noindent Recall the closed loop equation \eqref{eq:clloop} and the definition of the norm \defref{def:repnorm-0} and the sets $\s{W}$ and $\s{X}_{t}$. In the appendix, \lemref{lem:normproperties} summarizes some important properties of the norms $\Nxy{\anon}{\s{S}}$. We use these to obtain the following bounds on the one time-step growth of the state: 
\begin{lem}\label{lem:clineqs}
        Consider a state trajectory $(x_t)$ of the closed loop for a fixed $\s{W}$, then at each time step $t> \tau$ holds:
        \begin{subequations}
                \begin{align}
                        \Nxy{x_{t+1}}{\s{W}} &\leq \Nxy{W_{t-1}\tfrac{\lambda_{t-1}}{\Nxy{\lambda_{t-1}(x_t)}{1}}}{\s{W}}\Nxy{\lambda_{t-1}(x_t)}{1} + 1 \\
                        \label{eq:lambdineq}& \leq \Nxy{\lambda_{t-1}(x_t)}{1} + 1\\
                        \label{eq:normineq}& \leq \Nxy{x_t}{\s{X}_{t-1}} + 1 \\
                        \label{eq:normineqloose}& \leq \BXt{t-1} \Nxy{x_t}{\s{W}} + 1\\
                        \label{eq:kappaloose}&\leq \kappa_\tau\Nxy{x_t}{\s{W}} + 1
                \end{align}
        \end{subequations}   
\end{lem}

\noindent Recall, that the vector $\lambda_{t-1}(x_t)$ poses as a linear decomposition of $x_t$ in terms of the previous observations $X_{t-1}$, which is obtained through the minimization in \eqref{eq:cc2}. The right hand side of inequality \eqref{eq:normineq} and \eqref{eq:lambdineq} are equivalent. This follows from the equivalence relation
\begin{align}
        \Nxy{\lambda_{t-1}(\anon)}{1} = \Nxy{\anon}{\s{X}_{t-1}},
\end{align}
which follows from property \ref{it:norm-2} of \lemref{lem:normproperties} and is discussed in the appendix.

\noindent The inequality \eqref{eq:normineq} offers valuable insight into the closed loop behavior: The smaller $x_t$ is relative to the absolute convex hull of all previous observations $\s{X}_{t-1}$, the tighter the bound is on $\Nxy{x_{t+1}}{\s{W}}$. Hence, $\Nxy{x_t}{\s{X}_{t-1}}$ captures how well we can control a certain state $x_t$ given the observations made up until time $t$.

 If we rewrite $\Nxy{x_t}{\s{X}_{t-1}}$ as $\Nxy{x_t}{\s{W}}\Nxy{x_t/{\Nxy{x_t}{\s{W}}} }{\s{X}_{t-1}}$ and use the fact that the normalized vector $x_t/{\Nxy{x_t}{\s{W}}}$ lies in the set $\s{W}$, we obtain the looser upper-bound \eqref{eq:normineqloose}. Finally, \eqref{eq:kappaloose} is obtained by recalling that per definition $\Nxy{\s{W}}{\s{X}_t}$ is non-increasing in $t$ and therefore for all $t > \tau$ holds $\Nxy{\s{W}}{\s{X}_{t-1}} \leq \Nxy{\s{W}}{\s{X}_{\tau-1}} = \kappa_\tau$.

\subsection{Obtaining bounds on closed loop trajectories }

\noindent Recall the definition of a $\mu$-unstable transition in \defref{def:unstabletransition} and consider \lemref{lem:transitions}:
 If an $\mu$-unstable transition does not occur, \eqref{eq:V1leq} and \eqref{eq:V2leq} show that the quantities $V_1(x_t;\mu)$ and $V_2(x_t;\mu)$ do not increase for that time-step; On the other hand, \eqref{eq:V2geq} provides a bound on the increase of $V_2(x_t;\mu)$ if a $\mu$-unstable transition does occur.

\begin{lem}\label{lem:transitions}
        Let $(x_t)$ be a trajectory of \eqref{eq:clloop} with $t\geq 0$ and define the scalar functions $V_1(x;\mu) := \max\{0,\Nxy{x}{\s{W}}-\tfrac{1}{1-\mu}\}$ and $V_2(x;\mu) := \max\{\Nxy{x}{\s{W}},\tfrac{1}{1-\mu}\}$. Then,
        \begin{enumerate}[label = (\roman*)]
                \item \label{it:ineqstab} if $(x_{t+1},x_t) \notin \cl{U}_{\mu}$, then 
                \
        \begin{align} 
                \label{eq:V1leq}V_1(x_{t+1};\mu)  &\leq \mu V_1(x_{t};\mu)\\
                \label{eq:V2leq}V_2(x_{t+1};\mu) &\leq V_2(x_{t};\mu) 
        \end{align}
        \item \label{it:inequnstab}  if $(x_{t+1},x_t) \in \cl{U}_{\mu}$, then
        \begin{align} \label{eq:V2geq}
                V_2(x_{t+1};\mu) &\leq \kappa_t V_2(x_{t};\mu) + 1\\
                \label{eq:V1geq}
                V_1(x_{t+1};\mu) &> \mu V_1(x_{t};\mu) 
        \end{align}
        \end{enumerate}
        \end{lem}
        \begin{proof}
                see appendix.
        \end{proof}

\noindent The bounds of \thmref{thm:main} follows by combining the result of \thmref{thm:finiteunstable} with the above Lemma. To highlight the main proof techniques, we only focus on the derivation of \ref{it:limsup} and \ref{it:exp} of \thmref{thm:main} and refer to the appendix for a detailed proof of the remaining statements.\\

\noindent Consider some arbitrary closed loop trajectory $(x_t)$, fix $\tnull=0$ and some choice $\mu \in \cl{I}_{\kappa_0}$, where $\kk$ depends on the set $\s{W}$ and the initial guess matrix $X_{-1}=[\hat{x}_{1},\dots, \hat{x}_{n_0}]$. Recall, that $\kk$ measures the relative size between the disturbance set $\s{W}$ and the set $\cx{\s{X}_{-1}}$.
According to \thmref{thm:finiteunstable} the trajectory $(x_t)$ is guaranteed to have at most $N(\mu;\kappa_0)$-many $\mu$-unstable transitions. 
Hence, there is some finite time, call it $T'((x_t))$, such that for all time $t>T'((x_t))$ it holds $(x_{t+1},x_t) \notin \cl{U}_\mu$ and therefore the reverse inequality of \eqref{eq:unstabdef} holds, i.e.:
\begin{align}
        \Nxy{x_{t+1}}{\s{W}} \leq \max \left\{ \tfrac{1}{1-\mu}, \mu \Nxy{x_t}{\s{W}} +1  \right\},\quad \forall t>T'((x_t))
\end{align}
Now, apply the statement \ref{it:ineqstab} of \lemref{lem:transitions}, to conclude that for all $t>T'((x_t))$ holds $V_1(x_{t+1};\mu) \leq \mu V_1(x_{t};\mu)$. Therefore we get the convergence bound
\begin{align}\label{eq:exppreview}
        V_1(x_{T'((x_t))+k};\mu) \leq \mu^k V_1(x_{T'((x_t))};\mu), \quad k\geq 0
\end{align}
which proves that the trajectory $(x_t)$ has to be bounded. We also conclude that $\lim\limits_{t\rightarrow \infty}V_1(x_{t};\mu) = 0$ which leads to the asymptotic bound \begin{align}\label{eq:limsuppreview}
        \limsup\limits_{t \rightarrow \infty} \Nxy{x_t}{\s{W}} \leq \limsup\limits_{t \rightarrow \infty} (V_1(x_{t};\mu)+\tfrac{1}{1-\mu}) = \tfrac{1}{1-\mu}.
\end{align}

\noindent Similar type of arguments are used to derive the other statements of \thmref{thm:main} and are presented in appendix.

\section{Proving Finite Occurrence of Unstable transitions}\label{sec:boundingproof}
\noindent Here, we will discuss the key steps of proving \thmref{thm:finiteunstable}. The general idea will be to first argue that if an unstable transition occurred at time $t'$ and state $x_{t'}$, (i.e. $(x_{t'+1},x_{t'}) \in \c{U}_\mu$) then any future unstable transitions $(x_{t+1},x_t)\in \c{U}_\mu$, $t>t'$ must originate from some state $x_t$ which is significantly different from $x_{t'}$; In a second step, we then prove that there is a finite upper bound on how many significantly "different" unstable transitions can occur in the same trajectory, which leads to the result presented in \thmref{thm:finiteunstable}. In the following derivations we will make use of various simple facts from convex geometry, which are summarized in the appendix, \secref{sec:normballs}.

Matching the presentation of the theorem, in the derivations we will use the constant $\kk:=\BXtnull$ corresponding to some fixed set $\s{W}$, trajectory $(x_t)$ of the closed loop \eqref{eq:clloop} and initial time $\tau$. All throughout the discussion, $\mu$ will represent some fixed value in the open interval 
\begin{align}\label{eq:muinterval}
        \c{I}_\kk:=\Big( \Big(\sqrt{\tfrac{1}{4}+\tfrac{1}{\kk}}+\tfrac{1}{2}\Big)^{-1} , 1\Big)
\end{align}
and $\dd$ will refer to the corresponding transformed variable $\dd:=\tfrac{\mu^2}{1-\mu}\tfrac{1}{\kk}$, which always satisfies $\dd>1$. The following one-to-one relationship between both constants $\mu$ and $\dd$ will be frequently used and can be easily verified:
\begin{subequations}
        \label{eq:deltaandmu}
\begin{align}
        &&\dd &= \tfrac{\mu^2}{1-\mu}\tfrac{1}{\kk}, \quad\text{for }\mu \in \Big( \Big(\sqrt{\tfrac{1}{4}+\tfrac{1}{\kk}}+\tfrac{1}{2}\Big)^{-1} , 1\Big)\\
        \Leftrightarrow&&\mu &= \Big(\sqrt{\tfrac{1}{4} + \tfrac{1}{\dd \kk}} + \tfrac{1}{2} \Big)^{-1},\quad 
        \text{for }\dd \in (1,\infty) 
\end{align}
\end{subequations}
Our argument can be structured into the following three statements, which we prove separately in the next sections:

\begin{enumerate}[label=(\alph*)]
        \item We can radially project the set $\c{X}_\mu$ onto the ball $\tfrac{\dd}{\mu} \s{W}$ and show that the resulting set, called $\c{P}_\mu$, has the same cardinality as $\c{X}_\mu$.
        \item The set $\c{P}_\mu$ forms a $\dd$-separated subset of $\tfrac{\dd}{\mu} \s{W}$ with respect to a particularly chosen distance function $\dxyK{\anon}{\anon}{\s{X}_{\tnull-1}}$.
        \item There are some constants $c$ and $C$, such that for any $\dd$-separated subset $\s{P}$ of $\tfrac{\dd}{\mu} \s{W}$ we can construct a superset $\s{P} \subset \c{N}(\s{P})$ in $\R^n$ whose volume can be bounded above and below as $|\s{P}|c_\mathrm{in} \leq \vol(\c{N}(\s{P})) \leq C_{\mathrm{out}}$; Hence, the cardinality of any $\dd$-separated set, $\c{P}_\mu$ included, is bounded above by $\tfrac{C_{\mathrm{out}}}{c_\mathrm{in}}$.
\end{enumerate}

\subsection{Projection onto the ball $\tfrac{\dd}{\mu} \s{W}$}
\noindent Define the projection $\Pi_{\mu}:\R^n\mapsto\tfrac{\dd}{\mu} \s{W}$ as $\Pi_{\mu}(p):= \tfrac{\dd}{\mu\Nxy{p}{\s{W}}}p$ and define $\c{P}_\mu((x_t);\tnull)$ as the set resulting from applying $\Pi_{\mu}$ to every point in $\c{X}_\mu\left( (x_t);\tnull\right)$:
\begin{align}\label{eq:Pdef}
        \c{P}_\mu((x_t);\tnull):= \left\{\; \Pi_{\mu}(x_t) \;\left|\;x_t\in \c{X}_\mu( (x_t);\tnull)  \right. \right\}.
\end{align}
\begin{rem*}
To limit the notational burden, we will state the explicit dependency on the trajectory $(x_t)$ and $\tnull$ only in lemmas and theorems. For derivations, we will just write $\c{X}_\mu$, $\c{P}_\mu$ instead of $\c{P}_\mu((x_t);\tnull)$, $\c{X}_\mu((x_t);\tnull)$. 
\end{rem*}
\noindent Per construction, for every point $p \in \c{P}_\mu$ holds $\Nxy{p}{\s{W}} = \tfrac{\dd}{\mu}$ and therefore each $p \in \c{P}_\mu$ lies on the surface of the ball $\tfrac{\dd}{\mu} \s{W}$. 

\noindent Recall that for a time instance $t$, where $x_t \in \c{X}_\mu$ holds 
\begin{subequations}
        \begin{align}
                \label{eq:lbnxy}\Nxy{x_{t+1}}{\s{W}} &> \max \left\{ \tfrac{1}{1-\mu}, \mu \Nxy{x_t}{\s{W}} +1  \right\}\\
                \notag \Nxy{x_{t+1}}{\s{W}} &\leq \Nxy{x_t}{\s{X_{t-1}}} + 1, \\
                \label{eq:ubnxy}&\leq \BXt{t-1} \Nxy{x_t}{\s{W}} + 1
        \end{align}   
\end{subequations}
where \eqref{eq:lbnxy} is due to the definition of the set $\c{X}_\mu$ and \eqref{eq:ubnxy} follows from \lemref{lem:clineqs}. Combining the above inequalities, we can further establish that any $x_t \in \c{X}_\mu$ also satisfies the inequalities \eqref{eq:xtlb1}: 
\begin{lem}\label{lem:lb}
        \begin{subequations}
        \begin{align}
        \label{eq:xtlb1} &\mu\Nxy{x_t}{\s{W}} > \tfrac{\mu^2 }{(1-\mu)} \tfrac{1}{\BXt{t-1}}\\
        \label{eq:xtlb2}&\Nxy{\tfrac{1}{\mu\Nxy{x_{t}}{\s{W}}} {x_t}}{\s{X}_{t-1}} > 1. 
        \end{align}
        \end{subequations}
\end{lem}
\begin{proof}
Combining the lower-bound \eqref{eq:lbnxy} and the upper-bound \eqref{eq:ubnxy} yields
        \begin{align*}
                \Nxy{x_t}{\s{X}_{t-1}} + 1 &> \mu \Nxy{x_t}{\s{W}} +1 \\
                 \BXt{t-1} \Nxy{x_t}{\s{W}} &> \tfrac{1}{1-\mu}-1 = \mu \tfrac{1}{1-\mu}.
        \end{align*}  
\end{proof}

\noindent Using the above inequalities we can show in \lemref{lem:card} that $\c{P}_\mu$ has the same cardinality as $\c{X}_\mu$. Hence, instead of reasoning about the size of $\c{X}_\mu$ directly, we can equivalently study the size of the set $\c{P}_\mu$. As will become apparent in the following sections, the main advantage of analyzing the projected set $\c{P}_\mu$ rather than $\c{X}_\mu$ is that we can leverage $\c{P}_\mu$ as a subset of $\tfrac{\dd}{\mu} \s{W}$.
\begin{lem}\label{lem:card}
 $|\cl{X}_\mu| = |\cl{P}_\mu|$.
\end{lem}
\begin{proof}
From the definition of $\c{P}_\mu$ it is clear that $\c{P}_\mu$ has at most as many elements as $\c{X}_\mu$, hence trivially we have $|\cl{P}_\mu| \leq |\cl{X}_\mu|$. To establish $|\cl{P}_\mu| \geq |\cl{X}_\mu|$, we have to show that there are no two time instances $t_1 \neq t_2$ for which $x_{t_1}, x_{t_2} \in \c{X}_\mu$ gets mapped to the same point $p\in \cl{P}_\mu$. For the sake of proof by contradiction, assume for some $x_{t_1}, x_{t_2} \in \c{X}_\mu$ where w.l.o.g. $t_1<t_2$, holds $\dd (\mu \Nxy{x_{t_1}}{\s{W}})^{-1} x_{t_1} = \dd(\mu \Nxy{x_{t_2}}{\s{W}})^{-1} x_{t_2}$. Then, using \lemref{lem:lb} it follows:
\begin{align}
        \notag &\Nxy{\tfrac{1}{\mu \Nxy{x_{t_1}}{\s{W}}} x_{t_1}}{\s{X}_{t_2-1}} = \Nxy{\tfrac{1}{\mu \Nxy{x_{t_2}}{\s{W}}} x_{t_2}}{\s{X}_{t_2-1}}\overset{\eqref{eq:xtlb2}}{>} 1 \\
        \label{eq:ineqproofcard}\Rightarrow& \Nxy{x_{t_1}}{\s{X}_{t_2-1}} > \mu \Nxy{x_{t_1}}{\s{W}} \overset{\eqref{eq:xtlb1}}{>} \tfrac{\mu^2 }{(1-\mu)} \tfrac{1}{\BXt{t_1-1}}
\end{align}
Now, since $t_2>t_1$, it is clear that $x_{t_1} \in \cx{\s{X}_{t_2-1}}$ and therefore $\Nxy{x_{t_1}}{\s{X}_{t_2-1}} \leq 1$. Moreover, with \eqref{eq:ineqproofcard} and since $\mu$ is in the interval $\c{I}_{\kk}$, we are forced to conclude:
\begin{align}
        \BXt{t_1-1}> \frac{\mu^2 }{(1-\mu)}  \geq \BXtnull 
\end{align}
which is a contradiction, since $t_1\geq \tnull$ and we know that $\BXt{t}$ is non-increasing in $t$, 
\end{proof}

\subsection{Separateness of the set $\c{P}_\mu$ }

\noindent The previous section established, that the bounded set $\c{P}_\mu \subset \tfrac{\dd}{\mu} \s{W}$ has equal number of elements as $\c{X}_\mu$. Here, we will show that the points in the set $\cl{P}_\mu$ are "evenly spread out" across the surface of $\tfrac{\dd}{\mu} \s{W}$. Formally, we will term $\cl{P}_\mu$ to be a $\dd$-\textit{separated} subset of $\tfrac{\dd}{\mu} \s{W}$. This property will ultimately lead to the cardinality bound derived in the next section \ref{sec:volbound}.

The next lemma shows that any two points $p, p' \in \cl{P}_\mu$, $p \neq p'$ respect the inequality \eqref{eq:sep-p-ineq}.

\begin{lem}\label{lem:sepcondition-nometric}
        Let $\cl{P}_\mu((x_t);\tnull)$ be the projected set \eqref{eq:Pdef} and recall the definitions of the variables $\dd$, $\mu$ and $\kk$ in \eqref{eq:deltaandmu}. Then, for any two distinct points $p_1, p_2 \in  \cl{P}_\mu((x_t);\tnull)$, $p_1\neq p_2$ holds:
        \begin{align}\label{eq:sep-p-ineq}
                \max\{\Nxy{p_2}{\s{X}_{\tnull-1} \cup \s{p}_1}, \Nxy{p_1}{\s{X}_{\tnull-1} \cup \s{p}_2}\}  > \rad
        \end{align}
\end{lem}

\begin{proof}
Fix two arbitrary and distinct points $p_1,p_2 \in \c{P}_\mu$, $p_1\neq p_2$, then per definition of $\c{P}_\mu$ there are two corresponding elements $x_{t_1}, x_{t_2} \in \c{X}_\mu((x_t);\tnull)$ with $t_1 \neq t_2$ such that $p_1 = \dd(\mu \Nxy{x_{t_1}}{\s{W}})^{-1} x_{t_1}$ and $p_2 = \dd(\mu \Nxy{x_{t_2}}{\s{W}})^{-1} x_{t_2}$. We will prove the desired statement, by showing that depending on which unstable transition occurred first, i.e. $t_2>t_1$ or $t_1<t_2$, either $\Nxy{p_2}{\s{X}_{\tnull-1} \cup \s{p}_1}>\dd$ or $\Nxy{p_1}{\s{X}_{\tnull-1} \cup \s{p}_2}>\dd$ has to be satisfied. Inequality \eqref{eq:sep-p-ineq} then follows by taking the maximum of both cases. Hence, to complete our argument, w.l.o.g. we will assume the case $t_2>t_1$ and proceed to prove  $\Nxy{p_2}{\s{X}_{\tnull-1} \cup \s{p}_1}>\dd$; The case $t_1<t_2$ then follows by interchanging $t_1$ and $t_2$:\\
First, notice that since $\BXt{t_1-1} \leq \BXtnull$, we can conclude that for any $x_t\in \c{X}_{\mu}((x_t);\tnull)$, the following inequality is satisfied:
\begin{align}\label{eq:auxineq}
        \frac{\rad}{\mu \Nxy{x_t}{\s{W}}} \leq \frac{\mu^2 }{(1-\mu)} \frac{1}{\BXt{t-1}}\frac{1}{ \mu\Nxy{x_t}{\s{W}}} <1
\end{align}
Now, for $x_{t_2}$ recall from \eqref{eq:xtlb2} that 
\begin{align}\label{eq:sep-step1}
        \Nxy{p_2}{\s{X}_{t_2-1}} = \Nxy{\tfrac{\dd}{\mu\Nxy{x_{t_2}}{\s{W}}} {x_{t_2}}}{\s{X}_{t_2-1}} > \dd.
\end{align}
We will now use repeatedly the property \ref{it:norm-4} of \lemref{lem:normproperties}, to bound the left hand side of \eqref{eq:sep-step1} from above. To this end, consider first the following chain of inclusions:
\begin{align}
\notag\cx{\s{X}_{t_2-1}} \overset{a)}{\supset} \cx{\s{X}_{t_1}} &\overset{b)}{\supset} \cx{\s{X}_{\tnull-1} \cup \s{x}_{t_1}} \dots \\ 
\label{eq:sep-step2}&\quad \quad \dots\overset{c)}{\supset} \cx{\s{X}_{\tnull-1} \cup  \s{p}_1}
\end{align}
The inclusions $a)$, $b)$ follow directly from the definition of $X_{t}$. For inclusion $c)$, observe that $ p_1 = \rad (\mu \Nxy{x_t}{\s{W}})^{-1} x_t$ and recall from \eqref{eq:auxineq} that the scalar constant $\rad (\mu \Nxy{x_t}{\s{W}})^{-1}$ is less than one. Now, since $\cx{X_{\tnull} \cup \s{x}_{t_1}}$ is a symmetric convex body we know that it contains $0$. Hence we can view $ p_1$ as a convex combination of $0$ and $x_t$, which proves that $ p_ 1 \in \cx{\s{X}_{\tnull-1} \cup \s{x}_t}$ and therefore the set inclusion $c)$.

Now, using property \ref{it:norm-4} of \lemref{lem:normproperties} we can translate the inclusion \eqref{eq:sep-step2} into a chain of corresponding inequalities to bound the left hand side of \eqref{eq:sep-step1} and ultimately obtain:
$$\Nxy{p_2}{\s{X}_{\tnull-1} \cup \s{p}_1}>\dd.$$
\end{proof}
\noindent The term on the left hand side of inequality \eqref{eq:sep-p-ineq} can be seen as a binary operation $\dxyK{\anon}{\anon}{\s{X}_{\tnull-1}}$ on the points $p_1$ and $p_2$ which measures a particular notion of distance characterized by the symmetric convex body $\cx{\s{X}_{\tnull-1}}$. We will define this operation more generally for some set $\s{B}$ below and can use it to restate inequality \eqref{eq:sep-p-ineq} as $$\dxyK{p_1}{p_2}{\s{X}_{\tnull-1}}>\dd$$
\begin{defn}\label{def:xyK}
        Let $\s{B}$ be some bounded set in $\R^n$ and define the map $\dxyK{\anon}{\anon}{\s{B}} : \R^n \times \R^n \mapsto \R_{\geq 0}$ for each $x,y \in \R^n$ as 
        \begin{align}\label{eq:equivdxyK}
                \dxyK{x}{y}{\s{B}} := \min\left\{\;r\;\left|\begin{array}{c} \cx{\s{B} \cup \s{x}} \subset r \cx{\s{B} \cup \s{y}}\\
                        \cx{\s{B} \cup \s{y}} \subset r \cx{\s{B} \cup \s{x}} 
                \end{array} \right.\right\} 
        \end{align}
        or equivalently as
        \begin{align} 
                \dxyK{x}{y}{\s{B}} := \max\{\Nxy{\s{B} \cup \s{x}}{\s{B} \cup \s{y}},\Nxy{\s{B} \cup \s{y}}{\s{B} \cup \s{x}} \}
        \end{align}
\end{defn}
\noindent The definition of $d(\anon,\anon;\s{B})$ in the form of equation \eqref{eq:equivdxyK} gives a geometric intuition as to why the value $d(x,y;\s{B})$ can be viewed as a notion of distance between $x$ and $y$. As an example, consider in \figref{fig:distance} the two points $x,y \in\R^2$ which satisfy $d(x,y;\s{B})>r$ and where $\s{B}$ is taken as the box $[-1,1]\times [-1,1]$ in $\R^2$; Equation \eqref{eq:equivdxyK} then implies that $x$ lies outside of the set $r \cx{\s{B} \cup \s{y}}$ and $y$ lies outside of $r \cx{\s{B} \cup \s{x}}$. This scenario is presented in \figref{fig:distance} for $r=1.3$ and illustrates how the condition $d(x,y;\s{B})>r$ enforces a separation between $x$ and $y$. 
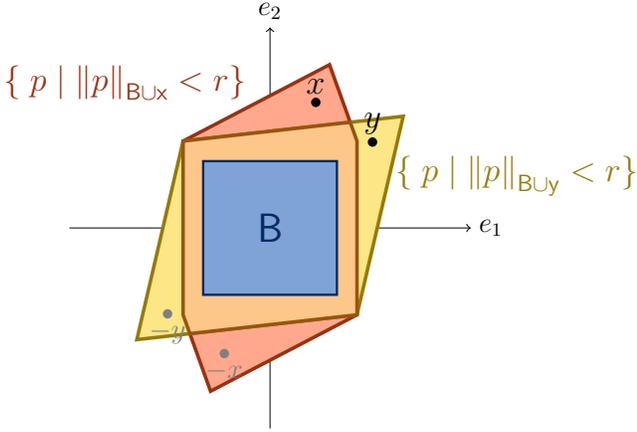
\begin{figure}
        \centering
        \resizebox{!}{0.35\textwidth}{
        \begin{tikzpicture}[font=\large] 

    \definecolor{K}{HTML}{0047AB}
    \definecolor{Kcopy}{HTML}{66B2DD}
    \definecolor{blowup}{HTML}{FF511C}
    \definecolor{cone}{HTML}{FACC0E}
    \definecolor{ballpair}{HTML}{00C855}
    \colorlet{ball}{blowup!50!cone}

    \def \light{ 50 }
    \def \dark{ 60 }
    \def \lthick{ 1.5pt }
    \def \lthin{ 1pt }
    
    \def \tta{ -10.00000000000000 } 
    \def \k{    -3.00000000000000 } 
    \def \l{     6.00000000000000 } 
    \def \d{     5.00000000000000 } 
    \def \h{     7.00000000000000 } 
    \def \r{     1.0000 } 
    \def \lam{   2.0 } 
    \def \xlen{   6.0 }
    \def \ylen{   6.0 }
    \def \pang{    70.0 } 
    \def \prad{   2.0}
    \def \qang{    40.0 } 
    \def \qrad{   2.0}
    \def \del{   1.3 } 
    \pgfmathparse{1/\del}
                \let\invdel\pgfmathresult

    \def \beta{ 20.0}

    \clip (-4.0,-3.5) rectangle (6.0,3.5);
    \coordinate (A) at (-\r,-\r); 
    \coordinate (B) at (\r,-\r); 
    \coordinate (C) at (\r,\r); 
    \coordinate (D) at (-\r,\r); 
    \coordinate (O) at (0,0);
    \coordinate (P) at ({\prad*cos(\pang)},{\prad*sin(\pang)});
    \coordinate (Pn) at ({-\prad*cos(\pang)},{-\prad*sin(\pang)});
    \coordinate (Q) at ({\qrad*cos(\qang)},{\qrad*sin(\qang)});
    \coordinate (Qn) at ({-\qrad*cos(\qang)},{-\qrad*sin(\qang)});
    \coordinate (A2) at ({-\lam*\r},{-\lam*\r}); 
    \coordinate (B2) at ({\lam*\r},{-\lam*\r}); 
    \coordinate (C2) at ({\lam*\r},{\lam*\r}); 
    \coordinate (D2) at ({-\lam*\r},{\lam*\r}); 
    \coordinate (xleft) at ({-0.5*\xlen},0);
    \coordinate (xright) at ({0.5*\xlen},0);
    \coordinate (ydown) at (0,{-0.5*\ylen});
    \coordinate (yup) at (0,{0.5*\ylen});
    \coordinate (P1) at (intersection of O--P and C--D);
    \coordinate (P2) at (intersection of O--P and C2--D2);
    \coordinate (Pn1) at (intersection of O--Pn and A--B);
    \coordinate (Pn2) at (intersection of O--Pn and A2--B2);
    \coordinate (Pinv) at ($\invdel*(P2)$);
    \coordinate (Ctop) at ($(C)!10!(Pinv)$);
    \coordinate (Dtop) at ($(D)!10!(Pinv)$);

    \fill[gray]    (Q) circle [radius=2pt]; 

    \draw[->] (xleft)--(xright) node[right]{$e_1$};
    \draw[->] (ydown)--(yup) node[above]{$e_2$};





    \coordinate (IX) at (intersection of Q--D and C--P);

    \filldraw[line width = \lthick, fill = blowup!\light, draw = blowup!\dark!black] 
          ($\del*(A)$) -- ($\del*(Pn)$) -- 
          ($\del*(Pn)$) -- ($\del*(B)$) -- 
          ($\del*(B)$) -- ($\del*(C)$) -- 
          ($\del*(C)$) -- ($\del*(P)$) -- 
          ($\del*(P)$) -- ($\del*(D)$) -- cycle;

    \filldraw[line width = \lthick, fill = cone!\light, draw = cone!\dark!black ] 
          ($\del*(D)$) -- ($\del*(Qn)$) -- 
          ($\del*(Qn)$) -- ($\del*(B)$) -- 
          ($\del*(B)$) -- ($\del*(Q)$) -- 
          ($\del*(Q)$) -- ($\del*(D)$);

      \filldraw[line width = \lthick, fill = ball!\light, draw = ball!\dark!black] 
      ($\del*(A)$) -- ($-\del*(IX)$) -- 
      ($-\del*(IX)$) -- ($\del*(B)$) -- 
      ($\del*(B)$) -- ($\del*(C)$) -- 
      ($\del*(C)$) -- ($\del*(IX)$) -- 
      ($\del*(IX)$) -- ($\del*(D)$) -- cycle;

    \draw[line width = \lthin, fill = K!\light, draw=K!\dark!black]
    (A) -- (B) -- (C) -- (D) -- cycle;


  \fill[black]    (P) circle [radius=2pt]; 
  \draw (P) node [above,black]     {\Large $x$};
  \fill[black]    (Q) circle [radius=2pt];
  \draw (Q) node [above,black]     {\Large $y$};
  \fill[gray]    (Pn) circle [radius=2pt]; 
  \draw (Pn) node [below,gray]     {$-x$};
  \fill[gray]    (Qn) circle [radius=2pt]; 
  \draw (Qn) node [below,gray]     {$-y$};
    \begin{scope}
    \draw (O) node [line width = \lthick, K!\dark!black]     {\LARGE $\s{B}$};
    \end{scope}

    \draw ($(Q)+(0.2,-0.5)$) node [line width = \lthick, right,cone!\dark!black] {\Large $\{\;p\;|\; \Nxy{p}{\s{B} \cup \s{y}} < r \}$};

    \draw ($(P)+(-0.9,0.3)$) node [line width = \lthick, left,blowup!\dark!black] {\Large $\{\;p\;|\; \Nxy{p}{\s{B} \cup \s{x}} < r \}$};

\end{tikzpicture}
        }
        \caption{Geometry of the distance function $d(\anon,\anon;\s{B})$: The points $x$ and $y$ satisfy the inequality $d(x,y;\s{B})>r$ in $n=2$, where $\s{B}$ is taken as the two dimensional cube $\s{B}:=\{(x_1,x_2)|{\;|x_i|\leq 1 }\}$ and $r=1.3$.}
        \label{fig:distance}
\end{figure}
We will call $x$ and $y$ to be $(r;\s{B})$-separated.
More generally, we will introduce the following terminology:
\begin{defn}
        A set $\s{P} \subset \R^n$ is $(\eps;\s{B})$-separated (with suitable set $\s{B}$), if $d(p,p';\s{B})>\eps$ holds for any $p,p'\in \s{P}$, $p\neq p'$.
\end{defn}
In terms of the above definition, \lemref{lem:sepcondition-nometric} states that $\c{P}_\mu$ is a $(\dd;\s{X}_{\tnull-1})$-\textit{separated} subset of $\tfrac{\dd}{\mu} \s{W}$. 

\subsection{Bounding the cardinality of $\c{P}_\mu$ through volume bounds }\label{sec:volbound}
 \noindent In this section we will complete the proof of \thmref{thm:finiteunstable}, by showing that any $(\eps,\s{B})$-separated subset of some bounded set $\s{S}$ has to be a finite set. This argument then leads to the results in \thmref{thm:finiteunstable}, since $\c{P}_\mu$ is a $(\dd;\s{X}_{\tnull-1})$-\textit{separated} subset of $\tfrac{\dd}{\mu} \s{W}$.
 
To illustrate the general idea, assume we would like construct a $(\eps;\s{B})$-separated set $\s{P}=\{p_1,p_2,\dots,\}$, contained within some larger bounded set $\s{S}\in \R^2$. In particular, assume we start with some $p_1\in \s{S}$, pick $p_2 \in \s{S}$ such that $d(p_1,p_2;\s{B}) > \eps$ and proceed to select each $p_n$ s.t. $d(p_n,p_{k};\s{B}) > \eps$ holds for all previous $k<n$. As illustrated in \figref{fig:distance}, it becomes intuitively clear that any constructed $(\eps;\s{B})$-separated subset $\s{P}$ in $\s{S}$ has to have finite cardinality, as it becomes increasingly harder to find "enough" room for a new point $p_n\in\s{S}$  which respects the separation condition w.r.t to previous points $d(p_n,p_{k};\s{B}) > \eps$, $k<n$.


 In the next section we will show by means of a volumetric argument that this intuition extends to $n$-dimensions and leads to a cardinality bound on the set $\c{P}_\mu$.
Denote $\s{P}$ to represent some $(\dd;\s{X}_{\tnull-1})$-\textit{separated} subset of $\tfrac{\dd}{\mu} \s{W}$, i.e. not necessarily $\c{P}_\mu$. We will bound $|\s{P}|$ by first constructing a corresponding covering set $\c{N}(\s{P}) \supset \s{P}$ and then showing that the volume of $\c{N}(\s{P})$ is bounded below and above as
\begin{align*}
        |\s{P}|c_\mathrm{in} \leq \vol(\c{N}(\s{P})) \leq C_{\mathrm{out}}.
\end{align*}
with some constants $c_\mathrm{in}$, $C_{\mathrm{out}}$ independent of $\s{P}$. The desired cardinality bound then takes the form $|\s{P}| \leq \tfrac{C_{\mathrm{out}}}{c_\mathrm{in}}$. 

The next sections will discuss: 1) the set $\c{N}(\s{P})$, 2) establishing the lower bound $|\s{P}|c_\mathrm{in}$, 3) proving the upper bound $C_{\mathrm{out}}$ and 4) formulating the statement of \thmref{thm:finiteunstable}.

\subsubsection{Covering set $\c{N}(\s{P})$}

\noindent For some point $p\in\R^n$ and $\eps>1$, define $\s{N}(p;\eps,\s{B})$ to stand for the set 
\begin{align}\label{eq:NpdK}
        \s{N}(p;\eps,\s{B}):= \left\{p' \in \R^n\left|\dxyK{p}{p'}{\s{B}}\leq \eps  \right. \right\}.
\end{align} 
See \figref{fig:sketch2} as an example for the geometry of the set $\s{N}(p;\eps,\s{B})$ in $\R^2$ with $\s{B} = [-1,1]\times[-1,1]$.
For a set $\s{P}$, correspondingly define the set $\c{N}(\s{P})$ as the following union of sets
\begin{align}\label{eq:collection}
        \c{N}(\s{P}) := \medcup\limits_{p \in \s{P}} \s{N}(p;\rad^{\frac{1}{2}},\s{X}_{\tnull-1}).
\end{align}
$\c{N}(\s{P})$ is a cover of $\s{P}$, since it can be easily verified that $\c{N}(\s{P})\subset \s{P}$.
It can be easily seen that the map $\dxyK{\anon}{\anon}{\s{B}}$ inherits the following properties from \lemref{lem:K1K2prop}:
\begin{lem}\label{lem:multdist} For all $x,y,z \in \R^n$ holds:
        \begin{enumerate}[label=(\roman*)]
                \item\label{it:met-1} $\dxyK{x}{x}{\s{B}}  = 1$
                \item\label{it:met-2} $\dxyK{x}{y}{\s{B}}  =\dxyK{y}{x}{\s{B}} = \dxyK{y}{-x}{\s{B}}$ 
                \item\label{it:met-3} $\dxyK{x}{y}{\s{B}} \leq \dxyK{x}{z}{\s{B}} \dxyK{z}{y}{\s{B}} $
        \end{enumerate}
\end{lem}
\noindent As shown in corollary \ref{coro:seppack}, the property \ref{it:met-3} of lemma \ref{lem:multdist} can be used to show that for $(\dd,\s{B})$-separated sets $\s{P}$, the sets in the union \eqref{eq:collection} are pairwise disjoint and we can therefore evaluate the volume $\vol(\c{N}(\s{P}))$ as the sum:
\begin{align}\label{eq:volfam}
        \vol(\c{N}(\s{P}))= \sum_{p \in \s{P}} \vol\big(\s{N}(p;\rad^{\frac{1}{2}},\s{X}_{\tnull-1}) \big).
\end{align}
\begin{coro}[of \lemref{lem:multdist}]\label{coro:seppack}
        If for some $x,y \in \R^n$ holds $\dxyK{x}{y}{\s{B}} > \eps $, then $\s{N}(x;\eps^{\frac{1}{2}},\s{B}) \cap \s{N}(y;\eps^{\frac{1}{2}},\s{B}) = \emptyset$
        \end{coro}
        \begin{proof}
        For the sake of proving the statement through contradiction, assume that there was some point $z\in\s{N}(x;\eps^{\frac{1}{2}},\s{B}) \cap \s{N}(y;\eps^{\frac{1}{2}},\s{B})$. Then, we know that $z$ satisfies both $\dxyK{x}{z}{\s{B}} \leq \eps^{\frac{1}{2}}$ and $\dxyK{y}{z}{\s{B}} \leq \eps^{\frac{1}{2}}$. But from the property \ref{it:met-3} of \lemref{lem:multdist}, we also have to conclude $$\dxyK{x}{y}{\s{B}} \leq \dxyK{x}{z}{\s{B}}\dxyK{z}{y}{\s{B}}\leq \eps $$ which leads to the intended contradiction.
\end{proof}

\subsubsection{Lower bound on volume of $\c{N}(\s{P})$}
To lower-bound the quantity \eqref{eq:volfam}, we will make use of the following lemma:
\begin{lem}\label{lem:ball}
Let $x,y\in \R^n$, then if $y = x + (\eps-1)p$ for some $p\in \s{B}$ and $\eps>1$, then it holds $d(x,y;\s{B}) \leq \eps$.
\end{lem}
\begin{proof}
We need to prove $\Nxy{x}{\s{B} \cup \s{y}} \leq \eps$ and $\Nxy{y}{\s{B} \cup \s{x}} \leq \eps$. 
\underline{$\Nxy{y}{\s{B} \cup \s{x}} \leq \eps$}: From the triangle inequality, we obtain
$$\Nxy{y}{\s{B} \cup \s{x}}=\Nxy{x + (\eps-1)p}{\s{B} \cup \s{x}} \leq \Nxy{x}{\s{B} \cup \s{x}} + (\eps-1)\Nxy{p}{\s{B} \cup \s{x}}  $$
and using the fact that $x,p \in \cx{\s{B} \cup \s{x}}$ by the norm definition \eqref{eq:defrepnorm} we get $\Nxy{y}{\s{B} \cup \s{x}}\leq 1+(\eps-1) = \eps$.\\
\underline{$\Nxy{x}{\s{B} \cup \s{y}} \leq \eps$}: Rewrite $x$ as $x = \eps\left(\tfrac{1}{\eps}(y) + \tfrac{\eps-1}{\eps}(-p) \right)$ and notice that $-p,y \in \cx{\s{B}\cup\s{y}}$, which shows that $x \in \eps \cx{\s{B}\cup\s{y}}$. Hence, via the norm definition \eqref{eq:defrepnorm} we conclude $\Nxy{x}{\s{B} \cup \s{y}}\leq \eps$. 
\end{proof}
\begin{figure*}
        \centering
        \resizebox{!}{0.3\textwidth}{
        \begin{tikzpicture}[font=\large] 


    \definecolor{K}{HTML}{0047AB}
    \definecolor{Kcopy}{HTML}{66B2DD}
    \definecolor{blowup}{HTML}{FF511C}
    \definecolor{cone}{HTML}{FFBE00}
  
    \colorlet{ball}{blowup!50!cone}
    \def \light{ 50 }
    \def \dark{ 60 }
    \def \lthick{ 1.5pt }
    \def \tta{ -10.00000000000000 } 
    \def \k{    -3.00000000000000 } 
    \def \l{     6.00000000000000 } 
    \def \d{     5.00000000000000 } 
    \def \h{     7.00000000000000 } 
    \def \r{     1.0000 } 
    \def \lam{   2.0 } 
    \def \xlen{   11.0 }
    \def \ylen{   6.0 }
    \def \pang{    70.0 } 
    \def \prad{   1.0}
    \def \del{   1.6 } 
    \pgfmathparse{1/\del}
                \let\invdel\pgfmathresult

    \def \beta{ 20.0}

    \clip (-6,-3.5) rectangle (6,3.5);
    \coordinate (A) at (-\r,-\r); 
    \coordinate (B) at (\r,-\r); 
    \coordinate (C) at (\r,\r); 
    \coordinate (D) at (-\r,\r); 
    \coordinate (O) at (0,0);
    \coordinate (P) at ({\prad*cos(\pang)},{\prad*sin(\pang)});
    \coordinate (Pn) at ({-\prad*cos(\pang)},{-\prad*sin(\pang)});
    \coordinate (A2) at ({-\lam*\r},{-\lam*\r}); 
    \coordinate (B2) at ({\lam*\r},{-\lam*\r}); 
    \coordinate (C2) at ({\lam*\r},{\lam*\r}); 
    \coordinate (D2) at ({-\lam*\r},{\lam*\r}); 
    \coordinate (xleft) at ({-0.5*\xlen},0);
    \coordinate (xright) at ({0.5*\xlen},0);
    \coordinate (ydown) at (0,{-0.5*\ylen});
    \coordinate (yup) at (0,{0.5*\ylen});
    \coordinate (P1) at (intersection of O--P and C--D);
    \coordinate (P2) at (intersection of O--P and C2--D2);
    \coordinate (Pn1) at (intersection of O--Pn and A--B);
    \coordinate (Pn2) at (intersection of O--Pn and A2--B2);
    \coordinate (Q) at ($(O)!0.75!\beta:(\lam,0)$);
    \coordinate (Pinv) at ($\invdel*(P2)$);
    \coordinate (Ctop) at ($(C)!10!(Pinv)$);
    \coordinate (Dtop) at ($(D)!10!(Pinv)$);
    \coordinate (ConeLabelPlus) at (1.5,2.5);
    \coordinate (blowuplabel) at (1.8,-0.5);
    \coordinate (ballstartarrow) at ($(P2)-(\del-1,0)-(0.4,0) $);
    \coordinate (ballendarrow) at (-3,1);
    \coordinate (ballstartarrow2) at ($(Pn2)-(\del-1,0)+(-0.2,0.4) $);
    \coordinate (ballendarrow2) at ($ (ballendarrow) -(1.0,0.5)$);
    \coordinate (Kcopyarrowstart) at  ($(P2)+0.7*(\del-1,0) $);
    \coordinate (Kcopyarrowstart2) at  ($(Pn2)+0.5*(\del-1,0) $);
    \coordinate (Kcopylabel) at (2.8, 1);
    \coordinate (Kcopylabel2) at (2.8, -2);

    \fill[black]  (A) circle [radius=2pt]; 
    \fill[black]    (B) circle [radius=2pt]; 
    \fill[black]  (C) circle [radius=2pt]; 
    \fill[black]    (D) circle [radius=2pt]; 
    \fill[gray]    (P1) circle [radius=2pt]; 
    \fill[gray]    (Pn1) circle [radius=2pt]; 
    \fill[gray]    (Q) circle [radius=2pt]; 



    \filldraw[line width = \lthick, fill= cone!\light, draw = cone!\dark!black] 
      (Ctop) -- (Pinv) -- (Dtop);

    \filldraw[line width = \lthick, fill= cone!\light, draw = cone!\dark!black] 
      ($-1*(Ctop)$) -- ($-1*(Pinv)$) -- ($-1*(Dtop)$);

    \filldraw[line width = \lthick, fill = blowup!\light, draw = blowup!\dark!black] 
          ($\del*(A)$) -- ($\del*(Pn2)$) -- 
          ($\del*(Pn2)$) -- ($\del*(B)$) -- 
          ($\del*(B)$) -- ($\del*(C)$) -- 
          ($\del*(C)$) -- ($\del*(P2)$) --
          ($\del*(P2)$) -- ($\del*(D)$) -- cycle;

    \begin{scope}
      \clip ($\del*(A)$) -- ($\del*(Pn2)$) -- 
          ($\del*(Pn2)$) -- ($\del*(B)$) -- 
          ($\del*(B)$) -- ($\del*(C)$) -- 
          ($\del*(C)$) -- ($\del*(P2)$) --
          ($\del*(P2)$) -- ($\del*(D)$) -- cycle;

      \fill[line width = \lthick, fill = ball!\light, draw = ball!\dark!black ] 
      (Ctop) -- (Pinv) -- (Dtop);
      \fill[line width = \lthick, fill = ball!\light, draw = ball!\dark!black ] 
      ($-1*(Ctop)$) -- ($-1*(Pinv)$) -- ($-1*(Dtop)$);
      
    \end{scope}
    \draw[->] (xleft)--(xright) node[right]{$e_1$};
    \draw[->] (ydown)--(yup) node[above]{$e_2$};
    \filldraw[line width = \lthick, fill = K!\light, draw=K!\dark!black] 
          (A) -- (B) -- (C) -- (D) -- cycle;



    \newcommand{\transK}[2]{
      \fill[line width = \lthick, fill = Kcopy!\light, draw = Kcopy!\dark!black]
          ($#1-(#2,#2)$) rectangle ($#1+(#2,#2)$);
      }

    \transK{(P2)}{ {\r*(\del-1)} };
    \transK{-1*(P2)}{ {\r*(\del-1)} };


  \fill[black]    (P2) circle [radius=2pt]; 
  \draw (P2) node [above,black]     {\Large $x$};
  \fill[gray!50!black]    (Pn2) circle [radius=2pt]; 
  \draw (Pn2) node [below,gray!50!black]     { $-x$};

  \draw (O) node [K!\dark!black]     {\LARGE $\s{B}$};
  \draw (ConeLabelPlus) node [right, cone!\dark!black] 
     {\Large $\{p\;|\;\Nxy{\s{x}}{\s{B}\cup \s{p}} \leq \eps  \} $};
  \draw ($-1*(ConeLabelPlus)$) node [left, cone!\dark!black] 
  {\Large $\{p\;|\;\Nxy{x}{\s{B}\cup \s{p}} \leq \eps  \} $};
  \draw (blowuplabel) node [right, blowup!\dark!black] 
  {\Large $\{p\;|\;\Nxy{p}{\s{B}\cup \s{x}} \leq \eps  \} $};

  \draw (ballendarrow) node (Nr) [left,ball!\dark!black] {\Large $\s{N}(x;\eps,\s{B})$ };
  \draw [->, ball!\dark!black, line width = \lthick] (ballstartarrow2) 
  to [out=180, in=260] (Nr.south);
  \draw [->, ball!\dark!black, line width = \lthick] (ballstartarrow) 
  to [out=180, in=45] (Nr.north east);

 \draw (Kcopylabel) node (Kcr) [right ,Kcopy!\dark!black, line width = \lthick] {\Large $(\eps-1)\s{B}$ };
 \draw (Kcopylabel2) node (Kcr2) [right ,Kcopy!\dark!black, line width = \lthick] {\Large $(\eps-1)\s{B}$ };
 \draw [->, Kcopy!\dark!black, line width = \lthick] (Kcopyarrowstart) 
 to [out=0, in=135] (Kcr.north west);
 \draw [->, Kcopy!\dark!black, line width = \lthick] (Kcopyarrowstart2) 
 to [out=0, in=180] (Kcr2.west);
 
    \end{tikzpicture}
        }
        
        \caption{$\s{N}(x;\eps,\s{B})$ is the intersection of the sets $\{p|\: \Nxy{x}{\s{B} \cup \s{p}} \leq \eps \}$ and $\{p|\:\Nxy{p}{\s{B} \cup \s{x}} \leq \eps \}$ and contains two translates of the set $(\eps-1)\s{B}$. In the picture, $\s{B}$ is taken as the two dimensional cube $\s{B}:=\{(x_1,x_2)|{\;|x_i|\leq 1 }\}$ and $\eps = 1.6$.}
        \label{fig:sketch2}
\end{figure*}
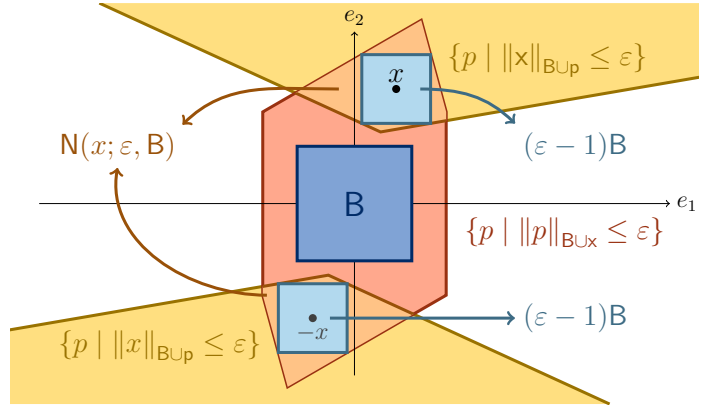

\noindent If we use the property \ref{it:met-2} of \lemref{lem:multdist}, then \lemref{lem:ball} tells us that each set $\s{N}(p;\eps,\s{B})$ contains the sets $\s{x}\oplus(\eps-1)\s{B}$ and $-\s{x}\oplus(\eps-1)\s{B}$, where the operator $\oplus$ denotes the Minkowski sum of two sets. For $\R^2$, \figref{fig:sketch2} illustrates the geometric relationship between the set $\s{N}(x;\eps,\s{B})$ and the set $\s{B}$ which is taken again to be the $\infty$-norm unit ball. We can see that the set $\s{N}(p;\eps,\s{B})$ is a union of two symmetrical polytopes and contains two non-overlapping translations (by the vector $x$ and $-x$) of the set $(\eps-1)\s{B}$. Now, using the fact that $n$-dimensional volume $\vol(\anon)$ is a homogenous function of degree $n$, we can obtain the following lower bound on the volume of any set $\s{N}(p;\eps,\s{B})$:
\begin{align}\label{eq:lbstep}
 \vol(\s{N}(p;\eps,\s{B})) \geq 2(\eps-1)^n \vol(\s{B}).
\end{align}
Combining this observation with our previous finding \eqref{eq:volfam}, we obtain the following lower bound on the volume of $\vol(\c{N}(\s{P}))$:
\begin{lem}\label{lem:lowerbound}
Let $\c{N}(\s{P})$ be the collection \eqref{eq:collection} corresponding to a $(\dd,\s{X}_{\tnull-1})$-separated ($\dd>1$) set $\s{P}$, then the volume $\vol(\c{N}(\s{P}))$ is bounded below by
\begin{align}\label{eq:Volnlower}
\vol(\c{N}(\s{P})) \geq 2(\rad^{\frac{1}{2}}-1)^n \vol(\cx{\s{X}_{\tnull-1}}) |\s{P}|,
\end{align}
 where $|\s{P}|$ denotes the cardinality of the set $\s{P}$.
\end{lem}
\begin{proof}
        Apply \eqref{eq:lbstep} to every term in the sum \eqref{eq:volfam}.
\end{proof}

\subsubsection{Upper bound on volume of $\c{N}(\s{P})$}

Consider some arbitrary point $q\in \s{N}(p;\dd^{\frac{1}{2}}, \s{X}_{\tnull-1})$ for some $p$ in the $\dd$-separated set $\s{P}$ and recall that $p \in \tfrac{\dd}{\mu}\s{W}$. Then, from the construction of the sets $\s{N}$ as \eqref{eq:NpdK} we can conclude that 
\begin{align}
\Nxy{q}{\s{X}_{\tnull-1} \cup p} \leq \dxyK{q}{p}{\s{X}_{\tnull-1}} \leq \dd^{\tfrac{1}{2}}.
\end{align}
Moreover, since we can upper bound $ \s{W}$ as $\s{W} \subset \kk \cx{\s{X}_{\tnull-1}}$, we also obtain  $\s{X}_{\tnull-1} \cup p \subset \max\{1, \tfrac{\kk\dd}{\mu}\}  \cx{\s{X}_{\tnull-1}}$ and therefore the point $q$ satisfies
\begin{align}
      \notag &&\Nxy{q}{\max\{1, \tfrac{\kk\dd}{\mu}\}  \s{X}_{\tnull-1}} \leq \Nxy{q}{\s{X}_{\tnull-1} \cup p} \leq \dd^{\frac{1}{2}} \\
      \label{eq:qpointin}\Leftrightarrow && q \in \dd^{\frac{1}{2}}\max\{1, \tfrac{\kk\dd}{\mu}\} \cx{\s{X}_{\tnull-1}}.
\end{align} 
Hence, \eqref{eq:qpointin} shows that the collection $\c{N}(\s{P})$ is a subset of $\dd^{\frac{1}{2}}\max\{1, \tfrac{\kk\dd}{\mu}\} \cx{\s{X}_{\tnull-1}}$ which proves the following upperbound on the volume $\vol(\c{N}(\s{P}))$:
\begin{lem}
        Let $\c{N}(\s{P})$ be the collection \eqref{eq:collection} corresponding to a $(\dd,\s{X}_{\tnull-1})$-separated ($\dd>1$) set $\s{P}$, then the volume $\vol(\c{N}(\s{P}))$ is bounded above by
        \begin{align}\label{eq:Volnupper}
                \vol(\c{N}(\s{P})) \leq \dd^{\frac{n}{2}}\max\{1,\tfrac{\kk\dd}{\mu}\}^n \vol(\cx{\s{X}_{\tnull-1}})
        \end{align}
\end{lem}
\subsubsection{Cardinality bound for $(\dd,\s{X}_{\tnull-1})$-separated sets $\s{P}$} 
Finally, the lower bound \eqref{eq:Volnlower} and upper bound \eqref{eq:Volnupper} imply the following bound on the cardinality of any $(\dd,\s{X}_{\tnull-1})$-separated set $\s{P}\subset\tfrac{\dd}{\mu} \s{W}$ with $\dd>1$ :
\begin{align}\label{eq:Pbound1}
|\s{P}| \leq \tfrac{1}{2}\left(\tfrac{\sqrt{\dd}}{\sqrt{\dd}-1}\right)^n \max\{1,\tfrac{\dd \kk}{\mu}\}^n
\end{align}
\figref{fig:sketch4} shows a pictorial summary of our derivation of the above inequality in $\R^2$. A $(\dd,\cx{\s{X}_{\tnull-1}})$-separated set $\s{P}=\{p_1,\dots,p_5\}$ is defined to satisfy $p_i \notin \s{N}(p_j,\dd,\s{X}_{\tnull-1})$, for all $i\neq j$ and as a consequence, we showed in \corref{coro:seppack} that the sets in the cover $\c{N}(\s{P})=\cup_{j}\s{N}(p_j,\sqrt{\dd},\s{X}_{\tnull-1})$,  are all disjoint. Then, \lemref{lem:ball} helped us establish that $\c{N}(\s{P})$ contains $2|\s{P}|$ many translations of the set $(\sqrt{\dd}-1) \cx{\s{X}_{\tnull-1}}$, which lead to the volume bound \eqref{eq:lbstep}. We obtain the upperbound \eqref{eq:Volnupper} by showing that $\c{N}(\s{P})$ has to be contained in the bigger box $\dd^{\frac{1}{2}}\max\{1, \tfrac{\kk\dd}{\mu}\} \cx{\s{X}_{\tnull-1}}$. So, in the context of the picture \figref{fig:sketch4}, we obtained our final cardinality bound \eqref{eq:Pbound1} by dividing the volume of the outer larger box by the volume of the smaller boxes.\\
\begin{figure}
        \resizebox{!}{0.33\textwidth}{
                \begin{tikzpicture}[font=\large] 

    \tikzset{invclip/.style={clip,insert path={{[reset cm]
    (-16383.99999pt,-16383.99999pt) rectangle (16383.99999pt,16383.99999pt)}}}}

    \definecolor{K}{HTML}{0047AB}
    \definecolor{Kcopy}{HTML}{66B2DD}
    \definecolor{blowup}{HTML}{FF511C}
    \definecolor{cone}{HTML}{FACC0E}
    \definecolor{ballpair}{HTML}{00C855}
    \colorlet{ball}{blowup!50!cone}

    \def \light{ 50 }
    \def \dark{ 60 }
    \def \lthick{ 1.5pt }
    \def \tta{ -10.00000000000000 } 
    \def \k{    -3.00000000000000 } 
    \def \l{     6.00000000000000 } 
    \def \d{     5.00000000000000 } 
    \def \h{     7.00000000000000 } 
    \def \r{     1.0000 } 
    \def \lam{   2.0 } 
    \def \xlen{   8.0 }
    \def \ylen{   7.0 }
    \def \pang{    70.0 } 
    \def \prad{   1.0}
    \def \del{   1.6 } 
    \pgfmathparse{1/\del}
                \let\invdel\pgfmathresult

    \def \beta{ 20.0}
    
    \clip (-3.65,-4) rectangle (8,4);
    \coordinate (A) at (-\r,-\r); 
    \coordinate (B) at (\r,-\r); 
    \coordinate (C) at (\r,\r); 
    \coordinate (D) at (-\r,\r); 
    \coordinate (O) at (0,0);
    \coordinate (P) at ({\prad*cos(\pang)},{\prad*sin(\pang)});
    \coordinate (Pn) at ({-\prad*cos(\pang)},{-\prad*sin(\pang)});
    \coordinate (A2) at ({-\lam*\r},{-\lam*\r}); 
    \coordinate (B2) at ({\lam*\r},{-\lam*\r}); 
    \coordinate (C2) at ({\lam*\r},{\lam*\r}); 
    \coordinate (D2) at ({-\lam*\r},{\lam*\r}); 
    \coordinate (xleft) at ({-0.5*\xlen},0);
    \coordinate (xright) at ({0.5*\xlen},0);
    \coordinate (ydown) at (0,{-0.5*\ylen});
    \coordinate (yup) at (0,{0.5*\ylen});
    \coordinate (P1) at (intersection of O--P and C--D);
    \coordinate (P2) at (intersection of O--P and C2--D2);
    \coordinate (Pn1) at (intersection of O--Pn and A--B);
    \coordinate (Pn2) at (intersection of O--Pn and A2--B2);
    \coordinate (Q) at ($(O)!0.75!\beta:(\lam,0)$);
    \coordinate (Pinv) at ($\invdel*(P2)$);
    \coordinate (Ctop) at ($(C)!10!(Pinv)$);
    \coordinate (Dtop) at ($(D)!10!(Pinv)$);
    \coordinate (ConeLabelPlus) at (1.5,2.5);
    \coordinate (blowuplabel) at (1.8,-0.5);
    \coordinate (ballstartarrow) at ($(P2)-(\del-1,0)-(0.4,0) $);
    \coordinate (ballendarrow) at (-3,1);
    \coordinate (ballstartarrow2) at ($(Pn2)-(\del-1,0)+(-0.2,0.4) $);
    \coordinate (ballendarrow2) at ($ (ballendarrow) -(1.0,0.5)$);
    \coordinate (Kcopyarrowstart) at  ($(P2)+0.7*(\del-1,0) $);
    \coordinate (Kcopyarrowstart2) at  ($(Pn2)+0.5*(\del-1,0) $);
    \coordinate (Kcopylabel) at (2.8, 1);
    \coordinate (Kcopylabel2) at (2.8, -2);
    
    \newcommand{\transK}[2]{
        \fill[thin, fill = Kcopy!\light, draw = Kcopy!\dark!black, fill opacity = 1]
            ($#1-(#2,#2)$) rectangle ($#1+(#2,#2)$);
        }

\def \dela{ 1.6}
\def \rpos{2.4}
\pgfmathparse{sqrt(\dela)}\let\sqrtdel\pgfmathresult
\pgfmathparse{sqrt(\dela)*\rpos}\let\normdel\pgfmathresult

\newcommand{\NrK}[6]{
\begin{scope}
\coordinate (X) at #1;
\coordinate (rX) at ($#4*(X)$);
\coordinate (Xinv) at (${1/#4}*(X)$);
\coordinate (L) at (intersection of rX--$#4*(#2)$ and Xinv--#3);
\coordinate (R) at (intersection of rX--$#4*(#3)$ and Xinv--#2);
\filldraw[#6] (Xinv) -- (L) -- (rX) -- (R) -- cycle ;
\filldraw[#6] ($-1*(Xinv)$) -- ($-1*(L)$) -- ($-1*(rX)$) -- ($-1*(R)$) -- cycle ;
\fill[black]    (X) circle [radius=2pt];  
\draw (X) node [above,black]     {\Large #5};
\end{scope}
}

\draw[->] (xleft)--(xright) node[right]{$e_1$};
\draw[->] (ydown)--(yup) node[above]{$e_2$};
\filldraw[thin, fill = K!\light, draw=K!\dark!black, fill opacity = 0.05]
($\normdel*(A)$) -- ($\normdel*(B)$)  -- ($\normdel*(C)$)
  -- ($\normdel*(D)$) -- cycle;

\filldraw[line width = \lthick, fill = K!\light, draw=K!\dark!black] 
(A) -- (B) -- (C) -- (D)  -- cycle;
\draw (O) node [K!\dark!black]     {\large $\cx{\s{X}_{\tnull-1}}$};

\draw ($(\normdel,\normdel)+(-1.5,0.5)$ ) node (outerboxnode) [right ,K!\dark!black, line width = \lthick] {\Large $\dd^{\frac{1}{2}}\max\{1, \tfrac{\kk\dd}{\mu}\} \cx{\s{X}_{\tnull-1}}$ };

\def \anga{ 160.0}
\def \angb{ 130}
\def \angc{ 90}
\def \angd{ 50}
\def \ange{ 20}

\NrK{(\anga:\rpos)}{A}{D}{\sqrtdel}{$p_1$}{line width = \lthick, fill= ballpair!\light, draw = ballpair!\dark!black};
\NrK{(\angb:\rpos)}{C}{A}{\sqrtdel}{$p_2$}{line width = \lthick, fill= ball!\light, draw = ball!\dark!black};
\NrK{(\angc:\rpos)}{C}{D}{\sqrtdel}{$p_3$}{line width = \lthick, fill= ball!\light, draw = ball!\dark!black};
\NrK{(\angd:\rpos)}{B}{D}{\sqrtdel}{$p_4$}{line width = \lthick, fill= ball!\light, draw = ball!\dark!black};
\NrK{(\ange:\rpos)}{B}{C}{\sqrtdel}{$p_5$}{line width = \lthick, fill= ball!\light, draw = ball!\dark!black};

\transK{(\anga:\rpos)}{ \sqrtdel-1 }; \transK{(\anga:-\rpos)}{ \sqrtdel-1 };
\transK{(\angb:\rpos)}{ \sqrtdel-1 }; \transK{(\angb:-\rpos)}{ \sqrtdel-1 };
\transK{(\angc:\rpos)}{ \sqrtdel-1 }; \transK{(\angc:-\rpos)}{ \sqrtdel-1 };
\transK{(\angd:\rpos)}{ \sqrtdel-1 }; \transK{(\angd:-\rpos)}{ \sqrtdel-1 };
\transK{(\ange:\rpos)}{ \sqrtdel-1 }; \transK{(\ange:-\rpos)}{ \sqrtdel-1 };

\NrK{(\anga:\rpos)}{A}{D}{\dela}{$p_1$}{ thin, fill= ballpair!\light, draw = ballpair!\dark!black, fill opacity = 0.1};
\NrK{(\angb:\rpos)}{C}{A}{\dela}{$p_2$}{ thin, fill= ball!\light, draw = ball!\dark!black, fill opacity = 0.1};
\NrK{(\angc:\rpos)}{C}{D}{\dela}{$p_3$}{ thin, fill= ball!\light, draw = ball!\dark!black, fill opacity = 0.1};
\NrK{(\angd:\rpos)}{B}{D}{\dela}{$p_4$}{ thin, fill= ball!\light, draw = ball!\dark!black, fill opacity = 0.1};
\NrK{(\ange:\rpos)}{B}{C}{\dela}{$p_5$}{ thin, fill= ball!\light, draw = ball!\dark!black, fill opacity = 0.1};

\coordinate (aoutlabelstart) at ($(\anga:-\rpos) + (0.7, +0.1)$);
\coordinate (aoutlabelend) at (4.4,-1);
\draw (aoutlabelend) node (balloutnode) [right, ballpair!\dark!black, line width = \lthick] {\Large $\s{N}(p_1;\dd,\s{X}_{\tnull-1})$ };
\draw [->, ballpair!\dark!black, line width = \lthick] (aoutlabelstart) 
to [out=-45, in=180] (balloutnode.west);

\coordinate (alabelstart) at ($(\anga:-\rpos) + (0.3, 0)$);
\coordinate (alabelend) at (3.8,-1.8);
\draw (alabelend) node (ballnode) [right ,ballpair!\dark!black, line width = \lthick] {\Large $\s{N}(p_1;\sqrt{\dd},\s{X}_{\tnull-1})$ };
\draw [->, ballpair!\dark!black, line width = \lthick] (alabelstart) 
to [out=-60, in=180] (ballnode.west);

\coordinate (alabelboxstart) at ($(\anga:-\rpos) + (0.0, 0)$);
\coordinate (alabelboxend) at (3.2,-2.7);
\draw (alabelboxend) node (ballboxnode) [right ,Kcopy!\dark!black, line width = \lthick] {\Large $(\sqrt{\dd}-1)\cx{\s{X}_{\tnull-1}}$ };
\draw [->, Kcopy!\dark!black, line width = \lthick] (alabelboxstart) 
to [out=-70, in=180] (ballboxnode.west);



\newcommand{\test}[3]{
\begin{scope}
    \def\myarr{($#1*(A)$)/($#1*(B)$), ($#1*(B)$)/($#1*(C)$), ($#1*(C)$)/($#1*(D)$), ($#1*(D)$)/($#1*(A)$) }
    \foreach \x/\y in \myarr{
        \fill [ fill = #3] \x -- ($#1*#2$) -- \y -- cycle;
        \fill [ fill = #3]  \x -- ($-#1*#2$) -- \y -- cycle;
    }
\end{scope}
}

\end{tikzpicture}
                }
                
                \caption{The geometry of a $(\dd,\cx{\s{X}_{\tnull-1}})$-separated set $\s{P}=\{p_1,\dots,p_5\}$. All sets are depicted geometrically accurate, assuming $\cx{\s{X}_{\tnull-1}}$ is the box $[-1,1]\times[-1,1]$ and $\dd=1.6$. }
                \label{fig:sketch4}
\end{figure}
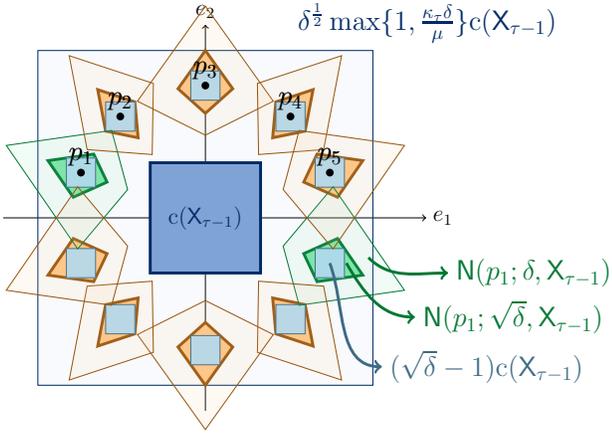

Recalling the relationship between the variables $\dd$, $\mu$ and $\kk$ in \eqref{eq:deltaandmu}, we can express $\mu$ in terms of some $\dd>1$ and $\kk$ as
\begin{align}
\mu = \dd \kk  \left(\sqrt{\tfrac{1}{4}+\tfrac{1}{\dd \kk}} - \tfrac{1}{2} \right) = \left(\sqrt{\tfrac{1}{4} + \tfrac{1}{\dd \kk}} + \tfrac{1}{2} \right)^{-1}
\end{align}
and can equivalently rewrite \eqref{eq:Pbound1} in terms of constant $\kk$ and $\dd>1$ as a free variable :
\begin{align}\label{eq:Pbound2}
        |\s{P}| \leq \tfrac{1}{2}\left(\tfrac{\sqrt{\dd}}{\sqrt{\dd}-1}\right)^n \max\{\tfrac{1}{\dd \kk}, \sqrt{\tfrac{1}{4} + \tfrac{1}{\dd \kk}} + \tfrac{1}{2} \}^n (\dd \kk)^{n}
        \end{align}

\noindent In summary, \eqref{eq:Pbound2} establishes a bound on the cardinality of $|\s{P}|$ which serves as an upper-bound on $|\c{P}_\mu((x_t),\tnull)| = |\c{X}_\mu((x_t),\tnull)|$, thus the total number of unstable transitions $\c{U}_\mu$ that can occur in the interval $[\tnull,\infty)$ of any closed loop trajectory $(x_t)$. We conclude by restating the results \thmref{thm:finiteunstable} again in terms of $\dd$:
\begin{thm*}
        For any trajectory $(x_t)$ of the closed loop \eqref{eq:clloop} and any $\tnull \geq 0$, the cardinality $|\c{X}_\mu\left((x_t);\tnull\right)|$ of the set $\c{X}_\mu\left((x_t);\tnull\right)$ is finite for any $\mu$ chosen as
        \begin{align}\label{eq:murange-thm-2}
                \mu = \left(\sqrt{\tfrac{1}{4} + \tfrac{1}{\dd \kk}} + \tfrac{1}{2} \right)^{-1}, \dd>1
        \end{align} 
        for some $\dd>1$ and bounded above as $|\c{X}_\mu\left((x_t);\tnull\right)|\leq N(\dd;\kk)$, where $N$ stands for the function
        \begin{align}\label{eq:Nbound-thm-2}
                {N}(\dd;\kk):=  \tfrac{1}{2}\left(\tfrac{\sqrt{\dd}}{\sqrt{\dd}-1}\right)^n \max\{\tfrac{1}{\dd \kk}, \sqrt{\tfrac{1}{4} + \tfrac{1}{\dd \kk}} + \tfrac{1}{2} \}^n (\dd \kk)^{n}
        \end{align}
        and $\kk$ is a constant computed from $\s{X}_{\tnull-1}$ as:
        \begin{align}\label{eq:kappa-2}
        \kk = \BXtnull := \max\limits_{z \in \s{W}} \Nxy{z}{\s{X}_{\tnull-1}}
        \end{align}
\end{thm*}

        

\section{A connection between metric entropy bounds and model-free stability analysis}\label{sec:metricentropy}
\noindent The notion of metric entropy\footnote{ this is to be distinguished from the Kolmogorov-Sinai entropy of a dynamical system introduced in \cite{kolmogorov1958new}} dates back to early works of A.N. Kolmogorov \cite{kolmogorov1959varepsilon} in 1959 and more recently it has been proven useful to study stochastic processes in the field of high-dimensional statistics. As an example, chap. 5 of \cite{wainwright_2019} discusses how bounds on the metric entropy of a metric space can be leveraged to obtain probabilistic bounds on the supremum of sub-gaussian processes over that same metric space; Particularly in machine learning applications, these mathematical results can then be used to derive learning theoretic guarantees of algorithms.\\

\noindent Reexamining the line of arguments that lead to our theoretical guarantees suggests that there might be a possibly fruitful connection between metric entropy bounds and worst-case performance bounds in the context of learning and control problems.
In retrospect, the main technique for stability analysis can be described as representing the collection $\c{X}_\mu$ of unstable transitions as a packing set $\s{P}_\mu$ in the \textit{totally bounded} metric space $(\tfrac{\dd}{\mu} \s{W},d(\anon,\anon;\s{X}_{\tnull-1}))$; The bound on the cardinality $\c{X}_\mu$ presented in \thmref{thm:finiteunstable} and derived in \secref{sec:boundingproof} can be viewed as the corresponding metric entropy bound. 
In hindsight, this inspires a new potential approach to algorithm design for learning and control: Synthesizing a control law for which unstable transitions (potentially as broader defined than $\c{U}_\mu$ considered here) form a packing in some totally bounded metric space. Similar to our presented result, we could also hope that smaller metric entropy translates to improved closed loop performance guarantees.\\

\noindent We will proceed by introducing metric entropy and related concepts based off of \cite{wainwright_2019} and \cite{Dudley2010}. In the later section, we will draw the connection to our stability analysis presented in \secref{sec:boundingproof}.



\subsection{Metric entropy of pseudo-metric spaces}
\noindent A pseudometric space $(\s{S},d)$ consists of a set $\s{S}$ and a pseudometric $d:\s{S}\times\s{S} \mapsto \R_{\geq 0}$, which satisfies the following properties:
\begin{enumerate}[label=(\roman*)]\label{def:metric}
\item \label{it:xx}$d(x,x) = 0$ for any $x\in\s{S}$
\item $d(x,y) = d(y,x)$ for any $x,y \in\s{S}$ 
\item $d(x,y) \leq d(x,y) + d(y,z)$ for any $x,y,z \in\s{S}$ 
\end{enumerate}
If in addition $d(x,y)=0$ holds only if $x=y$, then $d$ is called a metric and correspondingly $(\s{S},d)$ is a metric space.
\noindent The $\eps$-packing of $\s{S}$ w.r.t to $d$ is a set $\s{P} \subset \s{S}$ such that for each two distinct points $p_1,p_2 \in \s{P}$, $p_1 \neq p_2$ holds $d(p_1,p_2) >\eps$. 
Correspondingly, the $\eps$-packing number of $\s{S}$ is the cardinality of the largest $\eps$-packing set $\s{P}$ of $\s{S}$. Formally this is defined in \defref{def:packingnum} 
\begin{defn}\label{def:packingnum}
Let $(\s{S},d)$ be a metric (or pseudo-metric) space. Then the $\eps$-packing number $D(\s{S},\eps)$ (or $D(\s{S},\eps,d)$) of $\s{S}$ is defined as
\begin{align}
        D(\s{S},\eps) := \sup\left\{m\left|
        \begin{array}{c} \text{ for some } p_1,\dots,p_m \in \s{S},\\
                \;d(p_i,p_j) > \eps \text{ for }1\leq i<j\leq m 
        \end{array}\right. \right\}
\end{align}
\end{defn}
\noindent If $D(\s{S},\eps)$ is finite for any $\eps>0$, then $(\s{S},d)$ is often called \textit{totally bounded}. In this case we will define the quantity $\log(D(\s{S},\eps))$ as the \textit{metric entropy} of the set $\s{S}$ w.r.t. the metric $d$. 
\begin{rem}
An alternative definition of metric entropy as in \cite{wainwright_2019}, is $\log(N(\s{S},\eps))$ where $N(\s{S},\eps)$ is the covering number of the set $\s{S}$. The distinction between both (and other) definitions is of conventional matter, as it is well-known that packing numbers and covering numbers behave in equivalent manners. For our purpose, we will use the more fitting definition \eqref{def:packingnum}, which is for example used in the works of R.M. Dudley \cite{Dudley2010}.
\end{rem}

\subsection{Bounding occurrence of unstable transitions through metric entropy}
\noindent The key result behind our analysis in \secref{sec:boundingproof} was to show that any $(\dd,\s{X}_{\tnull-1})$-separated subset $\s{P}\subset \tfrac{\dd}{\mu} \s{W}$ respects the cardinality bound \eqref{eq:Pbound2}. From \lemref{lem:multdist} we can directly see that the operation $\log(d(\anon,\anon;\s{B})$ satisfies the properties of a pseudo metric \defref{def:metric} and therefore $(\tfrac{\dd}{\mu} \s{W}, \log(d(\anon,\anon;\s{X}_{\tnull-1}) )$ is a pseudo metric space. Correspondingly, the set $\s{P}\subset \tfrac{\dd}{\mu} \s{W}$ is $\log(\dd)$-packing in that same pseudo-metric space and our cardinality bound can be seen as an upper bound on the packing-number $D(\tfrac{\dd}{\mu} \s{W}, \log(\dd) )$ of the set $\tfrac{\dd}{\mu} \s{W}$ w.r.t. to the pseudometric $\log(d(\anon,\anon;\s{X}_{\tnull-1}) )$. Moreover, since we established the bound \eqref{eq:Pbound2} for every $\dd>1$ (or $\log(\dd)>0$), the space $(\tfrac{\dd}{\mu} \s{W}, \log(d(\anon,\anon;\s{X}_{\tnull-1}) )$ is a totally bounded pseudometric space with inequality \eqref{eq:Pbound2} implying a particular metric entropy bound. 

Hence in hindsight, our approach to stability analysis relied on mapping the set of unstable transitions $\cl{X}_\mu$ onto a fitting pseudometric space in which the metric entropy imposes a direct bound on the cardinality of the set $\cl{X}_\mu$. An interesting topic of further research is whether this general principle could be leveraged for model-free stability analysis and controller synthesis in broader learning and control problem settings.

\section{Simulation}\label{sec:sim}
\noindent We ran $N=1000$ simulations of the causal cancellation controller $\CC$ defined in \eqref{eq:cc2}. For the $k$th experiment, the trajectories $(x^k_t)$, $(u^k_t)$ are produced by the closed loop equations 
\begin{align}\label{eq:simsys}
        x^k_{t+1} &= A^k_0 x^k_t + \CCt{t}(x^k_t,X^{k+}_{t-1}, X^k_{t-1}, U^k_{t-1}) + w^k_t,
\end{align}
 and the system matrix $A^k_{0}\in \R^{3 \times 3}$, initial condition $x^k_0\in \R^3$ and disturbance $w^k_t$ is picked at random.

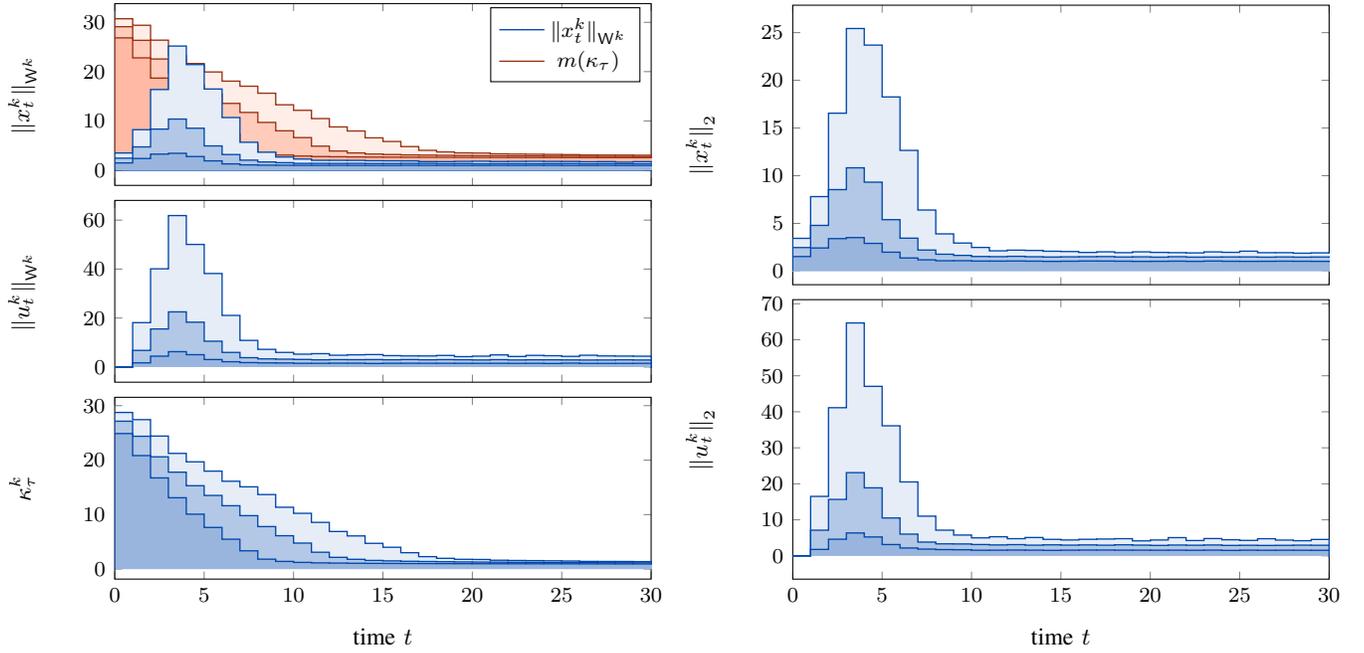
\begin{figure*}[ht]
        \begin{tikzpicture} 
                \tikzset{
                        small dot/.style={fill=black,circle,scale=0.3}
                }
                \definecolor{K}{HTML}{0047AB}
                \definecolor{Kcopy}{HTML}{66B2DD}
                \definecolor{blowup}{HTML}{FF511C}
                \definecolor{cone}{HTML}{FFBE00}
                \definecolor{ballpair}{HTML}{00C855}
                \colorlet{ball}{blowup!50!cone}
                \colorlet{K1}{K!10}
                \colorlet{K2}{K!20}
                \colorlet{K3}{K!30}
                \colorlet{K4}{K!40}
                \def \light{ 50 }
                \def \dark{ 60 }
                \def \opa{ 0.25 }
                \def \lin{ 0.5}
                \def \limi{30}
                 \def\iter{plotdata10000.99/10/0.5/solid,  plotdata10000.9/30/0.5/solid, plotdata10000.5/40/0.5/solid}
                \newcommand{\quant}[3]{
                        \foreach \x/\y/\l/\sty in \iter{
                                \edef \temp{
                                        \noexpand\addplot+[name path= q,
                                        const plot,
                                        draw=#2,
                                        mark = none,
                                         style = \sty,
                                        line width = \l,
                                        fill = #3!\y]
                                        table[x=t, y=#1, col sep=comma] {\x.csv}; %
                                }
                                \temp
                                }
                }
                \begin{groupplot}[
                    group style={
                        group name=my plots,
                        group size=1 by 3,
                        x descriptions at=edge bottom,
                        y descriptions at=edge left,
                        horizontal sep=0.2cm,
                        vertical sep=0.2cm,
                    },
                    footnotesize,
                    width=0.49\textwidth,
                    height=4cm,
                    xlabel=time $t$,
                    xmin = 0,
                    xmax = \limi,
                    xtick distance=5,
                ]
                \nextgroupplot[ylabel = $\Nxy{x^k_t}{\s{W}^k}$,
                         xmin = 0,
                         xmax = \limi]

                \addplot+[name path= zero,
                        draw=K,
                        mark color=K,
                        mark size = 0.0,
                        line width = 0.5,
                        domain=0:\limi ]
                coordinates {(40,0) (0,0)};

        \quant{mkappa}{blowup!\dark!black}{blowup}
        \quant{xbvec}{K}{K}
                \addlegendentry{$\Nxy{x^k_t}{\s{W}^k}$};
                \addlegendentry{$m(\kappa_{\tau})$};

                \nextgroupplot[ylabel = $\Nxy{u^k_t}{\s{W}^k}$]
                \quant{ubvec}{K}{K}

                \nextgroupplot[ylabel = $\kappa^k_{\tau}$]
                \quant{kappa}{K}{K}

                \end{groupplot}
            \end{tikzpicture}
            \begin{tikzpicture} 
                \tikzset{
                        small dot/.style={fill=black,circle,scale=0.3}
                }
                \definecolor{K}{HTML}{0047AB}
                \definecolor{Kcopy}{HTML}{66B2DD}
                \definecolor{blowup}{HTML}{FF511C}
                \definecolor{cone}{HTML}{FFBE00}
                \definecolor{ballpair}{HTML}{00C855}
                \colorlet{ball}{blowup!50!cone}
                \colorlet{K1}{K!10}
                \colorlet{K2}{K!20}
                \colorlet{K3}{K!30}
                \colorlet{K4}{K!40}
                \def \light{ 50 }
                \def \dark{ 60 }
                \def \opa{ 0.25 }
                \def \lin{ 0.5}
                \def \limi{30}
                 \def\iter{plotdata10000.99/10/0.5/solid, plotdata10000.9/30/0.5/solid, plotdata10000.5/40/0.5/solid}
                \newcommand{\quant}[3]{
                        \foreach \x/\y/\l/\sty in \iter{
                                \edef \temp{
                                        \noexpand\addplot+[name path= q,
                                        const plot,
                                        draw=#2,
                                        mark = none,
                                         style = \sty,
                                        line width = \l,
                                        fill = #3!\y]
                                        table[x=t, y=#1, col sep=comma] {\x.csv}; %
                                }
                                \temp
                                }
                }
                \begin{groupplot}[
                    group style={
                        group name=my plots,
                        group size=1 by 2,
                        x descriptions at=edge bottom,
                        y descriptions at=edge left,
                        horizontal sep=0.2cm,
                        vertical sep=0.2cm,
                    },
                    footnotesize,
                    width=0.49\textwidth,
                    height=5.3cm,
                    xlabel=time $t$,
                    xmin = 0,
                    xmax = \limi,
                    xtick distance=5,
                ]
                \nextgroupplot[ylabel = $\Nxy{x^k_t}{2}$,
                                xmin = 0,
                                xmax = \limi]

                \quant{x2norm}{K}{K}
        
                \nextgroupplot[ylabel = $\Nxy{u^k_t}{2}$]
                \quant{u2norm}{K}{K}


                \end{groupplot}
            \end{tikzpicture}
            \caption{Results of 1000 closed-loop simulations with random $A^k_0$, $x^k_0$ and disturbances $w^k_t$ drawn from $[-1,1]^3$. The plots on the left show the largest $1\%$, $10\%$, $50\%$ percentile values of 
                    $\Nxy{x^k_t}{\s{W}^k}$, $\Nxy{u^k_t}{\s{W}^k}$, $\kappa^k_\tau$, $m(\kappa_\tau)$. The right plot shows the same percentile values for the state $x^k_t$ and input $u^k_t$ measured in 2-norm.}
        \label{fig:Xcomboquant}
        \end{figure*}

\begin{figure*}[ht]
        \begin{tikzpicture} 
                        \tikzset{
                                small dot/.style={fill=black,circle,scale=0.3}
                        }
                        \definecolor{K}{HTML}{0047AB}
                        \definecolor{Kcopy}{HTML}{66B2DD}
                        \definecolor{blowup}{HTML}{FF511C}
                        \definecolor{cone}{HTML}{FFBE00}
                        \definecolor{ballpair}{HTML}{00C855}
                        \colorlet{ball}{blowup!50!cone}
                        \colorlet{K1}{K!10}
                        \colorlet{K2}{K!20}
                        \colorlet{K3}{K!30}
                        \colorlet{K4}{K!40}
                        \def \light{ 50 }
                        \def \dark{ 60 }
                        \def \opa{ 0.25 }
                        \def \lin{ 0.5}
                        \def \limi{60}
                         \def\iter{X0.99/10/0.5, X0.9/30/0.5, X0.5/40/0.5}
                        \newcommand{\quant}[3]{
                                \foreach \x/\y/\l in \iter{
                                        \edef \temp{
                                                \noexpand\addplot+[name path= q,
                                                const plot,
                                                draw=#2,
                                                mark = none,
                                                style = solid,
                                                line width = \l,
                                                fill = #3!\y]
                                                table[x=t, y=#1, col sep=comma] {nonoiseplotdata1000\x.csv}; %
                                        }
                                        \temp
                                        }
                        }
                        \begin{groupplot}[
                            group style={
                                group name=my plots,
                                group size=1 by 3,
                                x descriptions at=edge bottom,
                                y descriptions at=edge left,
                                horizontal sep=0.2cm,
                                vertical sep=0.2cm,
                            },
                            footnotesize,
                            width=0.49\textwidth,
                            height=4cm,
                            xlabel=time $t$,
                            xmin = 0,
                            xmax = \limi,
                            xtick distance=5,
                        ]
                        \nextgroupplot[ylabel = $\Nxy{x^k_t}{\s{W}^k}$,
                                 xmin = 0,
                                 xmax = \limi]

                        \addplot+[name path= zero,
                        draw=K,
                        mark color=K,
                        mark size = 0.0,
                        line width = 0.5,
                        domain=0:\limi ]
                coordinates {(40,0) (0,0)};
                \quant{mkappa}{blowup!\dark!black}{blowup}
                
                \quant{xbvec}{K}{K}
                        \addlegendentry{$\Nxy{x^k_t}{\s{W}^k}$};
                        \addlegendentry{$m(\kappa_{\tau})$};
        
                        \nextgroupplot[ylabel = $\Nxy{u^k_t}{\s{W}^k}$]
                        \quant{ubvec}{K}{K}
        
                        \nextgroupplot[ylabel = $\kappa^k_{\tau}$]
                        \quant{kappa}{K}{K}
        
                        \end{groupplot}
                \end{tikzpicture}
                \begin{tikzpicture} 
                        \tikzset{
                                small dot/.style={fill=black,circle,scale=0.3}
                        }
                        \definecolor{K}{HTML}{0047AB}
                        \definecolor{Kcopy}{HTML}{66B2DD}
                        \definecolor{blowup}{HTML}{FF511C}
                        \definecolor{cone}{HTML}{FFBE00}
                        \definecolor{ballpair}{HTML}{00C855}
                        \colorlet{ball}{blowup!50!cone}
                        \colorlet{K1}{K!10}
                        \colorlet{K2}{K!20}
                        \colorlet{K3}{K!30}
                        \colorlet{K4}{K!40}
                        \def \light{ 50 }
                        \def \dark{ 60 }
                        \def \opa{ 0.25 }
                        \def \lin{ 0.5}
                        \def \limi{60}
                         \def\iter{X0.99/10/0.5, X0.9/30/0.5, X0.5/40/0.5}
                        \newcommand{\quant}[3]{
                                \foreach \x/\y/\l in \iter{
                                        \edef \temp{
                                                \noexpand\addplot+[name path= q,
                                                const plot,
                                                draw=#2,
                                                mark = none,
                                                 style = solid,
                                                line width = \l,
                                                fill = #3!\y]
                                                table[x=t, y=#1, col sep=comma] {nonoiseplotdata1000\x.csv}; %
                                        }
                                        \temp
                                        }
                        }
                        \begin{groupplot}[
                            group style={
                                group name=my plots,
                                group size=1 by 2,
                                x descriptions at=edge bottom,
                                y descriptions at=edge left,
                                horizontal sep=0.2cm,
                                vertical sep=0.2cm,
                            },
                            footnotesize,
                            width=0.49\textwidth,
                            height=5.3cm,
                            xlabel=time $t$,
                            xmin = 0,
                            xmax = \limi,
                            xtick distance=5,
                        ]
                \nextgroupplot[ylabel = $\Nxy{x^k_t}{2}$,
                                xmin = 0,
                                xmax = \limi]

                \quant{x2norm}{K}{K}
        
                \nextgroupplot[ylabel = $\Nxy{u^k_t}{2}$]
                \quant{u2norm}{K}{K}

        
                        \end{groupplot}
                \end{tikzpicture}

                \caption{Results of 1000 closed-loop simulations with random $A^k_0$, $x^k_0$ and $w^k_t = 0$. The plots on the left show the largest $1\%$, $10\%$, $50\%^{*}$  percentile values of $\Nxy{x^k_t}{\s{W}^k}$, $\Nxy{u^k_t}{\s{W}^k}$, $\kappa^k_\tau$, $m(\kappa_\tau)$. ($^{*}$ this percentile is too small to visualize for $\Nxy{x^k_t}{\s{W}^k}$ and $\Nxy{u^k_t}{\s{W}^k}$)  The right plot shows the same percentile values for the state and input measured in 2-norm.}
                \label{fig:Xcomboquant_nonoise}
\end{figure*}
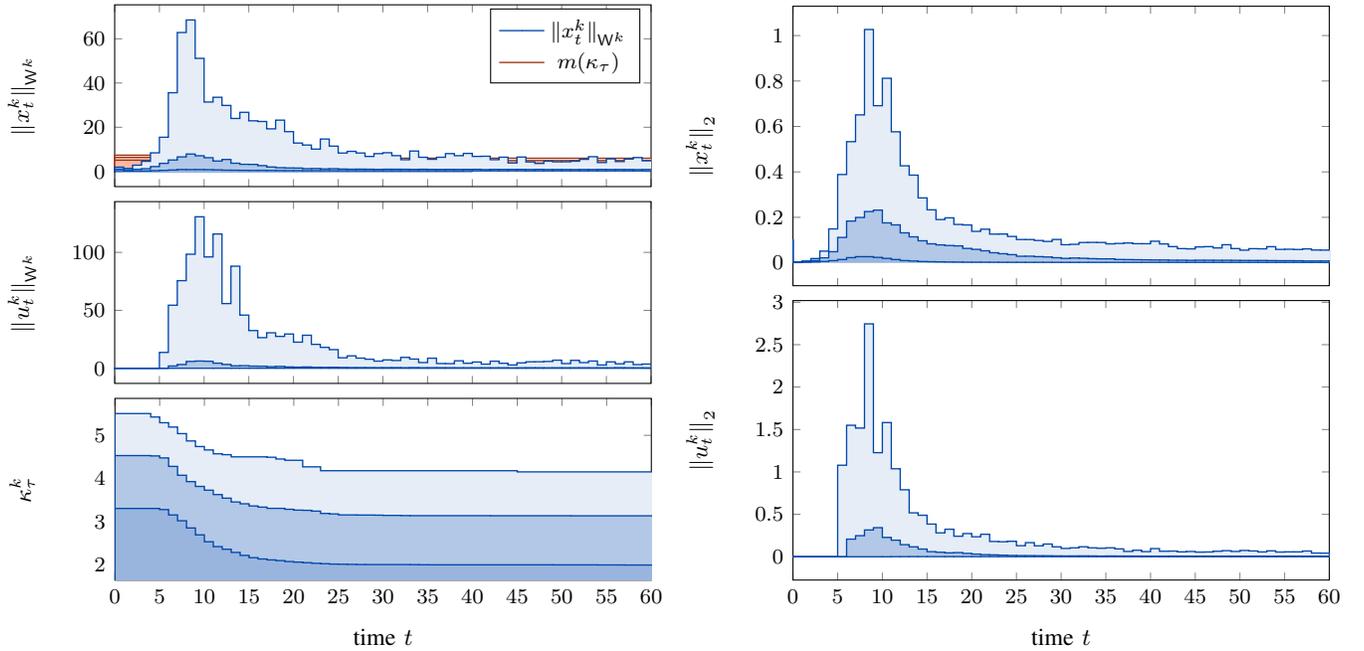

 All entries of $A^k_0$ and $x^k_0$ are picked i.i.d. from the standard gaussian distribution $ \mathcal{N}(0,1)$. In each experiment the causal cancellation controller \eqref{eq:cc2} is initialized as $X_{-1}=\eps I$, $X^+_{-1} = 0$, $U_{-1} = 0$ with the fixed choice $\varepsilon=0.1$. 

 \noindent \figref{fig:Xcombo_noise} shows simulation results of a single experiment
where $A_0$ is chosen as 
\begin{align}\label{eq:A0eig}
&A_0 = \begin{bmatrix}1.4 &0.2 &1 \\ 0.2 & 1.3 & 1 \\ 0.5 & 0.3 & 2 \end{bmatrix} & \lambda(A_0) =  \begin{bmatrix}2.7 \\ 1.13 \\ 0.86\end{bmatrix}.
\end{align}
and has a large unstable eigenvalue $\lambda_i(A_0)$. \figref{fig:Xcomboquant} and \figref{fig:Xcomboquant_nonoise} summarize the $N$ closed-loop experiments for two scenarios of disturbances. The graphs show, as a function of $t$, the largest $1\%$, $10\%$, $50\%$ percentiles of the values $\Nxy{x^k_t}{\s{W}^k}, \Nxy{u^k_t}{\s{W}^k}, \kappa^k_{t} = \Nxy{\s{W}^k}{\s{X}_{t-1}}$ among the $N$ experiments. In experiment $k$, the set $\s{W}^k$ is constructed according to equation \eqref{eq:Wdef} from the disturbance sequence\footnote{We assume that after $t>T$ the disturbance $w^k_t$ stays in the set $\cx{\s{W}^k}$} $(w^k_t)^{T-1}_{t=0}$ and the virtual disturbances $\hat{w}^k_i$. For our initialization of $X_{-1},X^+_{-1},U_{-1}$, the vectors $\hat{w}^k_i$ take the values $ -\eps A^k_0 e_i$, $1 \leq i \leq n$, where $e_i$ denotes the $i$th axis of the standard basis in $\R^n$:
\begin{align} 
        \s{W}^k = \left\{w^k_t|\: 0\leq t < T \right\} \cup \left\{-\varepsilon A^k_0 e_i|\: 1\leq i \leq n \right\}
\end{align}
 For \figref{fig:Xcomboquant}, the disturbance sequence $(w^k_t)$ of each experiment is picked i.i.d. uniformly from the interval $[-1,1]^3$. For \figref{fig:Xcomboquant_nonoise}, the initial condition $x_{0,i}$ is chosen i.i.d. according to the gaussian distribution $\mathcal{N}(0,\sigma^2)$, $\sigma=10^{-3}$ and $w^k_t = 0$. Since we have no disturbance, for this case, $\s{W}^k$ is simply the set of virtual disturbances $\{-\eps A^k_0 e_i|\: 1 \leq i \leq n \}$.

 We discussed that as a corollary of our main result (see \eqref{eq:mboundeq}), the causal cancellation controller $\CC$ guarantees for each experiment the asymptotic bound
 \begin{align}
\limsup\limits_{t \rightarrow \infty} \Nxy{x^k_t}{\s{W}^k} \leq m(\kappa^k_t), \quad \forall t.
 \end{align}
 where we take the function $m(\anon)$ to abbreviate the expression
 \begin{align} 
 m(s):= s\left(\tfrac{1}{2} + \sqrt{\tfrac{1}{4}+\tfrac{1}{s}}\right) +1.
 \end{align}
In \figref{fig:Xcomboquant} and \figref{fig:Xcomboquant_nonoise}, we overlayed the percentiles of $\Nxy{x^k_t}{\s{W}^k}$ (blue) and $m(\kappa^k_t)$ (red) to show that, qualitatively, the experiments match the above theoretical guarantee. In each experiment, the controller eventually learns to stabilize the unknown system (consistently after $10$ time-steps) and eventually (see $t>20$  in \figref{fig:Xcomboquant}) is bounded above by the asymptotic bound $m(\kappa^k_t)$. Notice also, that as more online data is observed, the asymptotic bound $m(\kappa^k_t)$ tightens. 

\figref{fig:Xcomboquant_nonoise} is showing the closed loop performance in the no disturbance regime. This is to investigate how the controller $\CC$ performs in absence of excitation by the disturbance. We see in \figref{fig:Xcomboquant_nonoise} that the controller $\CC$ stabilizes the system in all experiments, but in comparison to \eqref{fig:Xcomboquant}, we have a longer learning transient. Notice that in \figref{fig:Xcomboquant_nonoise}, the percentiles of the constant $\kappa^k_t$ do not decrease over time as much as in the experiments of \figref{fig:Xcomboquant}. Recall that for time $t$, the constant $\kappa^k_t$ can be seen to approximate the remaining uncertainty of the unknown system $A^k_0$. Hence, \figref{fig:Xcomboquant_nonoise} shows that despite remaining uncertainty in the system, the controller still manages to stabilize the system. This reflects that $\CC$ does not primarily care about identifying the unknown matrix $A^k_0$ but rather, collects \textit{only} enough data about the matrix $A^k_0$ to be able to stabilize the closed loop.

\begin{figure}[ht]
\begin{tikzpicture} 
        \tikzset{
                small dot/.style={fill=black,circle,scale=0.3}
        }
        \definecolor{K}{HTML}{0047AB}
        \definecolor{Kcopy}{HTML}{66B2DD}
        \definecolor{blowup}{HTML}{FF511C}
        \definecolor{cone}{HTML}{FFBE00}
        \definecolor{ballpair}{HTML}{00C855}
        \colorlet{ball}{blowup!50!cone}
        \def \light{ 50 }
        \def \dark{ 60 }

        \begin{groupplot}[
            group style={
                group name=my plots,
                group size=1 by 3,
                x descriptions at=edge bottom,
                y descriptions at=edge left,
                horizontal sep=0.2cm,
                vertical sep=0.2cm,
            },
            footnotesize,
            width=0.48\textwidth,
            height=3.5cm,
            xlabel=time $t$,
            xmin = 0,
            xmax = 20,
            xtick distance=5,
        ]
        \nextgroupplot[ylabel = $\Nxy{x_t}{\s{W}}$,
                 xmin = 0,
                 xmax = 40]
        \addplot+[const plot,
                draw=K,
                mark color=K,
                mark size = 0.5,
                line width = 1.2 ]
        table[x=t, y=xbvec, col sep=comma] {plotdata.csv}; 
        \addplot+[const plot,
                draw=blowup!\dark!black,
                mark color=blowup!\dark!black,
                mark size = 0.5,
                line width = 1.2]
                table[x=t, y=mkappa, col sep=comma] {plotdata.csv};
        \node at (axis description cs:0.9,0.28)  {\textcolor{blowup!\dark!black}{\small $\approx 2.66$}};
        \addlegendentry{$\Nxy{x_t}{\s{W}}$};
        \addlegendentry{$m(\kappa_{t})$};
        \nextgroupplot[ylabel = $\Nxy{u_t}{\s{W}}$]
        \addplot+[const plot,
                draw=K,
                mark color=K,
                mark size = 0.5,
                line width = 1.2]
        table[x=t, y=ubvec, col sep=comma] {plotdata.csv}; 
        \nextgroupplot[ylabel = $\kappa_{t}$]
        \addplot+[const plot,
                draw=K,
                mark color=K,
                mark size = 0.5,
                line width = 1.2]
        table[x=t, y=kappa, col sep=comma] {plotdata.csv}; 
          \node at (axis description cs:0.9,0.23)  {\textcolor{K}{\small $\approx 1.04$}};
        \end{groupplot}
    \end{tikzpicture}
    \caption{$\Nxy{x_t}{\s{W}}$ and $\Nxy{u_t}{\s{W}}$ trajectories for closed loop with uniform disturbance $w_{i,t}$ in $[-1,1]$ and $x_0 = [0.2,0,0.1]^T$.}
\label{fig:Xcombo_noise}
\end{figure}
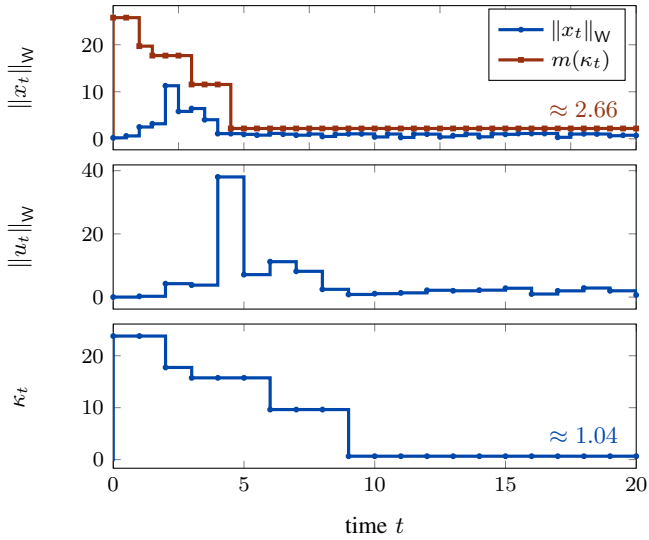

%
\section{Conclusion}
\noindent In this paper we derive a simple model-free controller that can adaptively and robustly stabilize a linear system with full actuation without any additional knowledge on disturbance, noise or parameter bounds. The controller comes with uniform asymptotic and worst-case guarantees on the state-deviation. The control design and stability analysis is enabled by a novel approach inspired by convex geometry, and simulations show that the controller is able to simultaneously learn and control the system in an efficient manner, even when applied to an open loop system with large unstable eigenvalues. Future work will further explore how this new perspective to adaptive control can provide more learning and control algorithms with robustness guarantees and nonrestrictive assumptions in a more general setting. In addition, we will investigate how the presented ideas can help in providing robustness and performance bounds for present methods in adaptive control and reinforcement learning. 
\bibliographystyle{IEEEtran}
\bibliography{IEEEabrv,refs}

\appendices 

\section{Preliminaries of convex geometry}\label{sec:appcvx}

\subsection{Convex bodies and norms}\label{sec:normballs}
\noindent It is a well known fact in convex geometry, that there is a one-to-one relationship between symmetric convex bodies (see \defref{def:symcvxbody}) and norms in $\R^n$. Here we will discuss this equivalence and how it relates to the properties of the $\Nxy{\anon}{\s{S}}$-norms which we used in this paper.
 The following discussion is adapted from chapter 1.7 of the standard text \cite{schneider2014convex}, to which we refer for more detail.

\begin{defn}[Symmetric Convex Body]\label{def:symcvxbody}
        A set $\s{B} \subset \R^n$ is a symmetric convex body if $\s{B}$ is a closed, bounded convex set with non-empty interior 
        and $z \in \s{B} \Leftrightarrow -z \in \s{B}$. 
\end{defn}

\noindent Symmetric convex bodies and norms are equivalent in $\R^n$, in the sense that any norm on $\R^n$ is uniquely defined by their corresponding unit norm ball and the space of all possible unit norm balls in $\R^n$ is precisely the set of symmetric convex bodies in $\R^n$. We summarize this in the following Lemma:
\begin{lem}\label{lem:normballsym}
For any norm $\Nxy{\anon}{}$ in $\R^n$, the corresponding norm ball $\left\{x \left|\Nxy{x}{} \leq 1 \right. \right\}$ is a symmetric convex set in $\R^n$. Conversely, for any symmetric convex body $\s{B}$ in $\R^n$, the function $g(\s{B},\anon):\R^n \mapsto \R^+$ defined by 
\begin{align}
        g(\s{B},x) := \min \left\{r \geq 0 \left| x \in r \s{B} \right.\right\},\quad \forall x \in \R^n
\end{align} 
is a norm on $\R^n$.
\end{lem}
\noindent In convex geometry, the function $g(\s{B},\cdot)$ is often called the \textit{gauge function} or the \textit{Minkowski functional} of $\s{B}$ and describes a concrete way to evaluate a norm based on knowing its unit ball. For our purposes it will be convenient to extend the above definition to derive norms from general bounded sets $\s{S} \subset \R^n$ in the following way: Given an arbitrary bounded set $\s{S}$, we will refer to $\Nxy{\anon}{\s{S}}$ as the norm $g(\cx{\s{S}},\anon)$, obtained by the Minkowski functional of the \textit{absolute convex hull} $\cx{\s{S}}$ of the set $\s{S}$. We defined this formally in \defref{def:abscvx-0} and \defref{def:repnorm-0}:
\begin{defn*}
        Let $\s{S}$ be a set in $\R^n$, then the set of all finite linear combinations $\sum^{N}_{i=1} \lambda_i x_i$ of elements $x_i$ in $\s{S}$ with $\sum^{N}_{i=1} |\lambda_i| \leq 1$ is called the \textit{absolute convex hull} of $\s{S}$ and we will refer to its closure as $\cx{\s{S}}$:
        \begin{align}\label{eq:defcxS-app}
        \cx{\s{S}} := \mathrm{cl}\left\{\sum^{N}_{i=1} \lambda_i x_i\left|\quad \{x_i\}^{N}_{i=1} \subset \s{S}, \quad \sum^{N}_{i=1} |\lambda_i| \leq 1 \right.\right\}.
        \end{align}
\end{defn*}
\begin{rem*}
Equivalently, $\cx{\s{S}}$ is the closure of the convex hull of the set $(-\s{S})\cup \s{S}$.
\end{rem*}
\begin{defn*}
        For a fixed bounded set $\s{S} \subset \R^n$, let $\Nxy{\anon}{\s{S}}:\R^n \mapsto \R_{\geq 0}$ be the norm defined for all $x\in\R^n$ as
\begin{align*}
        \Nxy{x}{\s{S}} := \left\{\begin{array}{ll} \min \left\{r \geq 0 \left| x \in r \cx{\s{S}}\right.\right\}, & \text{for }x\in \mathrm{span}(\s{S}) \\ \infty, &\text{else}
         \end{array}\right.
\end{align*}
and for sets $\s{S}' \subset \R^n$, define $\Nxy{\s{S}'}{\s{S}}$ as the quantity
\begin{align}\label{eq:s1s2}
        \Nxy{\s{S}'}{\s{S}} &:= \max\limits_{z \in \cx{\s{S}'}} \Nxy{z}{\s{S}}
\end{align}
\end{defn*}

\noindent The definition is overloading the common notation for the norm $\Nxy{\anon}{\s{S}}$ as used in \cite{schneider2014convex} where $\s{S}$ is required to be a symmetric convex body. Applying \defref{def:repnorm-0} to the disturbance set $\s{W}$, \eqref{eq:Wdef} defines the previously introduced norm $\Nxy{\anon}{\s{W}}$ which we use to formulate the stability analysis. \figref{fig:norm1} illustrates an example of how the set $\cx{\s{W}}$ and the norm $\Nxy{x}{\s{W}}$ are related in two dimensions.

\noindent The following Lemma summarizes some key properties following from the above definitions. Note property \ref{it:norm-3} and \ref{it:norm-4}, which show that we can verify set-membership and set inclusions of symmetric convex bodies in terms of the norm. Moreover, property \ref{it:norm-2} describes a practical evaluation of $\Nxy{\anon}{\s{S}}$ for finite sets $\s{S}$ and shows 
\begin{align}\label{eq:previewopt-x}
        \begin{array}{rll}
                \Nxy{\lambda_{t-1}(x)}{1}=\min\limits_{\lambda} & \left\| \lambda\right\|_{1} &=\Nxy{x}{\s{X}_{t-1}} \\
        \mathrm{s.t.}& X_{t-1} \lambda = x 
        \end{array}
\end{align}
\begin{lem}\label{lem:normproperties}
Norms according to Definition \ref{def:repnorm-0} satisfy:
\begin{enumerate}[label=(\roman*)]
        \item\label{it:norm-1} for any $\s{S}$ suiting \defref{def:repnorm-0}, $\cx{\s{S}}$ is a symmetric convex body in $\R^n$. Moreover, $\cx{\s{S}}$ is the unit norm ball of $\Nxy{\anon}{\s{S}}$, so we can equiv. write $\Nxy{\anon}{\s{S}} = \Nxy{\anon}{\cx{\s{S}}}$.
        \item\label{it:norm-2}  if $\s{S}$ is a finite set $\s{S}=\{p_1,\dots,p_N\}$, then for any $x \in \R^n$, $\Nxy{x}{\s{S}}$ can be computed as:
        \begin{align*}
                \Nxy{x}{\s{S}} = \min\left\{\sum^{N}_{i=1} |\lambda_i|\;\left| \lambda_1, \dots, \lambda_N \text{ s.t. }\sum^{N}_{i=1} \lambda_i p_i = x  \right. \right\}
        \end{align*}
        \item \label{it:norm-3} for all $x\in \R^n$ holds $x \in \cx{\s{S}}  \Leftrightarrow \Nxy{x}{\s{S}} \leq 1$
        \item \label{it:norm-4} $\cx{\s{S}_1} \subset \cx{\s{S}_2}$ holds if and only if for all $x\in \R^n$ holds  $\Nxy{x}{\s{S}_1} \geq \Nxy{x}{\s{S}_2}$.
        \item \label{it:norm-5} for all $\gamma >  0$, holds $ \Nxy{\anon}{\frac{1}{\gamma} \s{S}} = \gamma\Nxy{\anon}{\s{S}}$
\end{enumerate}
\end{lem}
\begin{proof}
        \noindent The statements of \lemref{lem:normproperties} are easy to verify: \ref{it:norm-1}, \ref{it:norm-3}, \ref{it:norm-4} follow directly from \lemref{lem:normballsym} and \ref{it:norm-2} follows by using the description of the set shown in \eqref{eq:defcxS-app} to rewrite the definition of $\Nxy{\anon}{\s{S}}$.
\end{proof}
\noindent We use the definition \eqref{eq:s1s2} of $\Nxy{\s{S}_1}{\s{S}_2}$ to measure the size of a set $\s{S}_1$ w.r.t. the norm $\Nxy{\anon}{\s{S}_2}$ of another set $\s{S}_2$. The following properties can be easily verified:
\begin{lem}\label{lem:s1s2distprop}
        Let $\s{S}_1$, $\s{S}_2$ be some bounded sets in $\R^n$ and recall definitions \defref{def:repnorm-0}. Then it holds
\begin{enumerate}[label=(\roman*)]
        \item\label{it:s1s2-1} The quantity $\Nxy{\s{S}_1}{\s{S}_2}$ can be equivalently defined as $\Nxy{\s{S}_1}{\s{S}_2}:= \min\{{\;t\;}|{\; \s{S}_1 \subset t \cx{\s{S}_2},\; t\geq 0 }\}$ 
        \item\label{it:s1s2-2} equivalence of norms in $\R^n$:
        \begin{align}
                \frac{1}{\Nxy{\s{S}_1}{\s{S}_2}}\Nxy{\anon}{\s{S_2}} \leq \Nxy{\anon}{\s{S_1}} \leq \Nxy{\s{S}_2}{\s{S}_1}\Nxy{\anon}{\s{S_2}}
        \end{align}
        \item\label{it:s1s2-3} $\Nxy{\s{S}_1}{\s{S}_2} \leq 1 \Leftrightarrow \s{S}_1 \subset \cx{\s{S}_2}$
        \item\label{it:s1s2-4} $\s{S}_1 \subset \Nxy{\s{S}_1}{\s{S}_2} \cx{\s{S}_2}$
\end{enumerate}
\end{lem}
\noindent Property \ref{it:s1s2-1} states that we can equivalently define $\Nxy{\s{S}_1}{\s{S}_2}$ as the smallest factor $t$ such that $\s{S}_1$ is contained in $t\cx{\s{S}_2}$. This definition is visualized in the middle plot of \figref{fig:norm1} for some exemplary sets $\s{W}$ and $\s{X}_1$ in $\R^2$. The other properties can be derived as immediate consequences of property \ref{it:s1s2-1}: \ref{it:s1s2-1} $\Rightarrow$ \ref{it:s1s2-4} $\Rightarrow$ \ref{it:s1s2-3}, \ref{it:s1s2-2}.
\subsection{Distance between norms}\label{sec:distofnorms}
\noindent For bounded sets $\s{S}_1$, $\s{S}_2$ we define $d(\s{S}_1,\s{S}_2)$ as a multiplicative distance $d(\s{S}_1,\s{S}_2)$ between the two norms $\Nxy{\anon}{\s{S}_1}$ and $\Nxy{\anon}{\s{S}_2}$: 
\begin{defn}\label{def:K1K2}
        Let $\s{S}_1,\s{S}_2 \subset \R^n$ be sets with norms $\Nxy{\anon}{\s{S}_1}$, $\Nxy{\anon}{\s{S}_2}$ defined as in \defref{def:repnorm-0}. Then, define $d(\s{S}_1,\s{S}_2)$ as
        \begin{align}
                \label{eq:ds1s2} d(\s{S}_1,\s{S}_2) &:= \max\{\Nxy{\s{S}_1}{\s{S}_2},\Nxy{\s{S}_2}{\s{S}_1} \}
        \end{align}
\end{defn}

\begin{lem}\label{lem:K1K2prop}
The definitions \ref{def:K1K2} imply
\begin{enumerate}[label=(\roman*)]
        \item\label{it:dist-1} $d(\s{S},\s{S}) = 1$
        \item\label{it:dist-2} $d(\s{S}_1,\s{S}_2)=d(\s{S}_2,\s{S}_1)$
        \item\label{it:dist-3} $d(\s{S}_1,\s{S}_2) \leq d(\s{S}_1,\s{S}')d(\s{S}',\s{S}_2)$
\end{enumerate}
\end{lem}
\begin{proof}
Statement \ref{it:dist-1} and \ref{it:dist-2} are trivial. Part \ref{it:dist-3} follows by using \ref{it:s1s2-4} and \ref{it:s1s2-1} of \lemref{lem:s1s2distprop}: 
Notice that 
\begin{align}
        \s{S}_1 &\subset d(\s{S}_1,\s{S}')\cx{\s{S}'} \subset d(\s{S}_1,\s{S}')d(\s{S}',\s{S}_2)\cx{\s{S}_2} \\
         \s{S}_2 &\subset d(\s{S}_2,\s{S}')\cx{\s{S}'} \subset d(\s{S}_2,\s{S}')d(\s{S}',\s{S}_1)\cx{\s{S}_1}
\end{align}
leads to $\max\{\Nxy{\s{S}_1}{\s{S}_2},\Nxy{\s{S}_2}{\s{S}_1} \}\leq d(\s{S}_2,\s{S}')d(\s{S}',\s{S}_1)$.\end{proof}
\noindent \lemref{lem:K1K2prop} shows that the map $d(\anon,\anon)$ can be viewed as a multiplicative distance between sets: The above properties imply that $\log d(\anon,\anon)$ is a pseudo-metric over the space of bounded sets in $\R^n$.

\section{Proofs}

\subsection{Theorem \ref{thm:main}}

\begin{proof}
        The proof follows by applying \lemref{lem:transitions}.The corresponding bounds for $(u_t)$ are then obtained by using the equation \eqref{eq:rewrite}.
        
        According to the setting of the theorem, consider some fixed trajectories $(x_t)$, $(u_t)$, reference time $\tnull$, $\mu \in \c{I}_\kk$ with $\kk=\BXtnull$.
        Then, as discussed before, a direct consequence of \thmref{thm:finiteunstable} is that there is some trajectory dependent finite time $T'<\infty$, such that in the time interval $[0,T']$ there are at most $N(\mu;\kk)$-many time-instances $\c{T} := \{t'_1,\dots,t'_M\}$, where
         \begin{align}
                \tnull\leq t'_1<t'_2<\dots<t'_{M}\leq T',\quad M<N(\mu;\kk)
         \end{align}
         at which $\mu$-unstable transitions occur and for all other time-instances $t \neq t'_i$ holds the opposite inequality of \eqref{eq:unstabdef}. Thus, depending on whether or not $t$ belongs to $\c{T}$, the transitions $(x_{t+1},x_t)$ of the trajectory $(x_t)$ satisfy
         \begin{subequations}
                \label{eq:summary}
         \begin{align}
                \Nxy{x_{t+1}}{\s{W}} > \max \left\{ \tfrac{1}{1-\mu}, \mu \Nxy{x_t}{\s{W}} +1  \right\},\quad \forall t\in\c{T}\\
                \Nxy{x_{t+1}}{\s{W}} \leq \max \left\{ \tfrac{1}{1-\mu}, \mu \Nxy{x_t}{\s{W}} +1  \right\},\quad \forall t\notin\c{T}.
         \end{align}
        \end{subequations}
        Moreover, combining \lemref{lem:transitions} with the above, we obtain that w.r.t. to the function $V_1(x;\mu) := \max\{0,\Nxy{x}{\s{W}}-\tfrac{1}{1-\mu}\}$, the transitions $(x_{t+1},x_t)$ respect the inequality 
        \begin{subequations}
                \label{eq:summaryV1}
         \begin{align}
                V_1(x_{t+1};\mu) &> \mu V_1(x_{t};\mu) ,\quad \forall t\in\c{T}\\
                V_1(x_{t+1};\mu) &\leq \mu V_1(x_{t};\mu) ,\quad \forall t\notin\c{T}
         \end{align}
        \end{subequations}
        and for function $V_2(x;\mu) := \max\{\Nxy{x}{\s{W}},\tfrac{1}{1-\mu}\}$, the transitions $(x_{t+1},x_t)$ satisfy
\begin{subequations}
        \label{eq:summaryV2}
        \begin{align}
        V_2(x_{t+1};\mu) &\leq \BXt{t-1}V_2(x_{t};\mu) + 1 ,\quad &\forall t\in\c{T}\\
        V_2(x_{t+1};\mu) &\leq V_2(x_{t};\mu) ,\quad &\forall t\notin\c{T}.
        \end{align}
\end{subequations}
        
        The bounds on $(x_t)$ in part \ref{it:limsup} and \ref{it:exp} were derived before from \eqref{eq:summaryV1}, (see \eqref{eq:exppreview} and \eqref{eq:limsuppreview}). For \ref{it:supbound}, notice that \eqref{eq:summaryV2} implies that for any $t$, $V_2(x_{t};\mu)$ can be bounded above by $V_2(x_{t'_M+1};\mu)$, since aside from the time intervals $[t'_i , t'_{i+1}]$, the quantity $V_2(x_{t};\mu)$ is guaranteed to be non-increasing.  Moreover \eqref{eq:summaryV2} also shows that $V_2(x_{t'_{M}+1};\mu)$ can be bounded above as
\begin{align} \label{eq:subboundstep}
& V_2(x_{t'_M+1};\mu) \leq \alpha V_2(x_{\tnull};\mu) + \beta\\
        \notag &\alpha := \prod\limits^{M}_{k=1}\BXt{t'_k-1},\quad 
        \beta := \sum^{M-1}_{k=0}\prod\limits^{k}_{j=1}\BXt{t'_{M-j}-1}.
\end{align}
        Recall that $\BXt{t}$ is non-increasing (hence $\BXt{t'_k-1}\leq \kk:=\BXt{\tnull-1}$) and the bound $M\leq N(\mu;\kk)$, to see that the constants $\alpha$ and $\beta$ are bounded above as
        \begin{align}
        &\alpha \leq \max\{1,\kk^{N(\mu;\kk)}\} &\beta \leq \frac{1-\kk^{N(\mu;\kk)}}{1-\kk}.
        \end{align}
        We then obtain the final inequality \eqref{eq:xsubbound} by substituting the above bounds into \eqref{eq:subboundstep} and observing that 
        \begin{align}
                \sup_{t\geq \tnull} \Nxy{x_t}{\s{W}}  \leq \sup_{t\geq \tnull}  V_2(x_{t};\mu) \leq V_2(x_{t'_M+1};\mu)
        \end{align}

                \noindent To obtain the corresponding bounds for the input $(u_t)$, recall that $u_t$ can be rewritten as
        \begin{align*}
                u_t  &= \left(U_{t-1} -X_{t:1} \right)\lambda_{t-1} (x_t)\\
                        &= (-A_0 X_{t-1} -W_{t-1})\lambda_{t-1} (x_t)\\
                        &=-A_0x_t-W_{t-1}\lambda_{t-1} (x_t)
        \end{align*}
                and that $\Nxy{\lambda_{t-1}}{1} = \Nxy{x_t}{\s{X}_{t-1}}$. This allows us to upper-bound $\Nxy{u_t}{\s{W}}$ by 
                \begin{align*} 
                        \Nxy{u_t}{\s{W}} &\leq \Nxy{A_0\frac{x_t}{\Nxy{x_t}{\s{W}}}}{\s{W}} \Nxy{x_t}{\s{W}} + \Nxy{\frac{x_t}{\Nxy{x_t}{\s{W}}}}{\s{X}_{t-1}}\Nxy{x_t}{\s{W}}\\ 
                        &\leq (\max\limits_{x \in \s{W}} \|A_0x\|_{\s{W}} + \Nxy{\s{W}}{\s{X}_{\tnull-1}}) \Nxy{x_t}{\s{W}}\\
                        &\leq (\|A_0\|_{\s{W}} + \kk)\Nxy{x_t}{\s{W}}
                \end{align*}
                and obtain desired bounds for $(u_t)$ by plugging in the already derived bounds for $(x_t)$. 
\end{proof}

\subsection{Lemma \ref{lem:transitions}}
\begin{proof}
        \underline{Part \ref{it:ineqstab} and \eqref{eq:V1geq}:} We can expand the inequality as
        \begin{align}\label{eq:convcomb}
                \Nxy{x_{t+1}}{\s{W}} \leq \max \left\{ \tfrac{1}{1-\mu}, \mu \Nxy{x_t}{\s{W}} +(1-\mu)\tfrac{1}{1-\mu}  \right\}
        \end{align}
        and can subtract $\tfrac{1}{1-\mu}$ on both sides to obtain 
        \begin{align*}
                &\Nxy{x_{t+1}}{\s{W}}-\tfrac{1}{1-\mu} \leq \max \left\{ 0, \mu (\Nxy{x_t}{\s{W}} -\tfrac{1}{1-\mu})  \right\}\\
                \Leftrightarrow & \max\{0,\Nxy{x_{t+1}}{\s{W}}-\tfrac{1}{1-\mu} \} \leq \mu \max \left\{ 0, \Nxy{x_t}{\s{W}} -\tfrac{1}{1-\mu}  \right\}
        \end{align*}
        Similarly, noticing that the second term on the right hand sight of \eqref{eq:convcomb} is a convex combination of $\Nxy{x_t}{\s{W}}$ and $\tfrac{1}{1-\mu}$, we can conclude 
        \begin{align*}
                \max\{\Nxy{x_{t+1}}{\s{W}},\tfrac{1}{1-\mu}\} \leq  \max\{\Nxy{x_{t}}{\s{W}},\tfrac{1}{1-\mu}\}.
        \end{align*}
        \underline{Part \ref{it:inequnstab}:} We previously derived, that the inequality \eqref{eq:normineqloose} holds for all time $t$:
        \begin{align*}
                \Nxy{x_{t+1}}{\s{W}} \leq \BXt{t-1} \Nxy{x_t}{\s{W}} + 1.
        \end{align*}
        Now, if in addition inequality \eqref{eq:unstabdef} holds, then we obtain
        \begin{align*}
                \max \left\{ \tfrac{1}{1-\mu}, \mu \Nxy{x_t}{\s{W}} +1  \right\} <\BXt{t-1} \Nxy{x_t}{\s{W}} + 1.
        \end{align*}
        Combining both the previous inequalities, we get
        \begin{align*}
                \max \left\{\tfrac{1}{1-\mu}, \Nxy{x_{t+1}}{\s{W}} \right\} &\leq \BXt{t-1} \Nxy{x_t}{\s{W}} + 1\\
                &\leq \BXt{t-1} \max\left\{\tfrac{1}{1-\mu},\Nxy{x_t}{\s{W}}\right\} + 1.
        \end{align*}
\end{proof}

\end{document}